\documentclass[leqno,final]{siamltex}
  \usepackage{graphicx} \graphicspath{ {./images/} }
  \usepackage{amsmath,amstext,amssymb,bm}
  \usepackage{leftidx}
  \numberwithin{equation}{section}
  \usepackage{xcolor}
  \usepackage{soul}
  \usepackage{tikz}
  \usetikzlibrary{shapes,arrows}
  \usepackage{mathrsfs}
  \usepackage{tikz-cd}
  
  \usepackage{hyperref}
  \usepackage{cleveref}
  
  \newtheorem{remark}{Remark}[section]

  \allowdisplaybreaks[4]
  
  \newcommand{\cd}{{\cdot}}  
  \newcommand{\p}{{\partial}}  
  \newcommand{\Del}{{\Delta}}
  \newcommand{\del}{{\delta}}
  \newcommand{\nab}{\nabla}
  \newcommand{\eps}{\epsilon}
  \newcommand{\gam}{\gamma}
  \newcommand{\Gam}{\Gamma}
  \newcommand{\be}{\beta}
  \newcommand{\al}{\alpha}
  \newcommand{\kap}{\kappa}
  
  \newcommand{\Os}{{\Omega^s}}
  \newcommand{\Om}{{\Omega^m}}
  \newcommand{\tP}{\widetilde{P}}
  \newcommand{\Pe}{P_e}
  \newcommand{\Da}{D_a}
  
  \newcommand{\R}{\mathbb{R}}

  \newcommand{\calA}{\mathcal{A}}
  \newcommand{\calB}{\mathcal{B}}
  \newcommand{\calRh}{\mathcal{R}_h}
  \newcommand{\calPh}{\mathcal{P}_h}
  \newcommand{\calO}{\mathcal{O}}
  \newcommand{\Linf}{L^\infty}
  \newcommand{\V}{\mathbb V}

  \renewcommand{\(}{\left(} 
  \renewcommand{\)}{\right)}
  \newcommand{\<}{\langle}
  \renewcommand{\>}{\rangle}
  \newcommand{\x}{\times}
  \newcommand{\weakto}{\rightharpoonup}
  \newcommand{\les}{\lesssim}
  \renewcommand{\nab}{\nabla}

   \graphicspath{ {./images/} }
  
\begin{document}

  	\title{Mathematical and Numerical Analysis for PDE Systems Modeling Intravascular Drug Release from Arterial Stents and Transport in Arterial Tissue\thanks{This work was partially supported by NSF 2012414 and NSF DMS-2309626.} }
  	\markboth{X. FENG and T. JIANG}{MODELING INTRAVASCULAR DRUG RELEASE FROM ARTERIAL STENTS}
  	
  	%\author{Xiaobing Feng\dag
  		%\address{Department of Mathematics, The University of Tennessee, Knoxville, TN 37996, U.S.A. }
  		%\email{xfeng@math.utk.edu}
  		%\thanks{\dag Department of Mathematics, The University of Tennessee, Knoxville, TN 37996, U.S.A. (xfeng@math.utk.edu).
  			%	The work of this author was partially supported by the NSF grant: DMS-1620168.}
  		%
  		%\author{Mitchell Sutton\ddag}
  		%\address{Department of Mathematics, The University of Tennessee, Knoxville, TN 37996, U.S.A. }
  		%\email{msutto11@vols.utk.edu}
  		%\thanks{\ddag Department of Mathematics, The University of Tennessee, Knoxville, TN 37996, U.S.A. (msutto11@vols.utk.edu).
  			%	The work of this author was partially supported by the NSF grant: DMS-1620168.}
  		
  		\author{Xiaobing Feng\thanks{Department of Mathematics, The University of Tennessee, 
  				Knoxville, TN 37996. U.S.A. (xfeng@utk.edu).}
  			\and{Tingao Jiang}\thanks{Department of Mathematics, The University of Tennessee, 
  				Knoxville, TN 37996. U.S.A. (tjiang12@utk.edu).} }
  		
  		\date{}
  		
  		\maketitle
  		
  		\thispagestyle{empty}
  		
  		\begin{abstract}
  			%This paper delves into the PDE and numerical analysis of a modified one-dimensional intravascular stent model, initially proposed in \cite{McGinty13}. The unique weak solution existence for the modified model is demonstrated through the application of the Galerkin method, complemented by a compactness argument. We present formulations for both a semi-discrete finite element method and a fully discrete scheme utilizing Euler time-stepping for the PDE model. Rigorous proof of optimal-order error estimates in the energy norm is provided for both schemes. The paper includes numerical results, featuring comparisons between various decoupling strategies and time-stepping schemes. Lastly, we briefly discuss extensions of the model to the two-dimensional case, presenting PDE and numerical analysis results for these extensions.
  			%
  			This paper is concerned with the PDE and numerical analysis of a modified one-dimensional  intravascular stent model originally proposed in \cite{McGinty13}.  It is proved that the modified model has a unique weak solution using the Galerkin method combined with a compactness argument. A  semi-discrete finite element method and a fully discrete scheme using the Euler time-stepping are formulated for the PDE model. Optimal order error estimates in the energy norm are proved for both schemes. Numerical results are presented along with comparisons between different decoupling strategies and time-stepping schemes. Lastly, extensions of the model and its PDE and numerical analysis results to the two-dimensional case are also briefly discussed.
  		\end{abstract}
  		
  		\begin{keywords}
  			Intravascular stent, drug release, diffusion-advection-reaction equation, well-posedness, Galerkin and finite element method, error estimates,  pharmacokinetics.
  		\end{keywords}
  		
  		\begin{AMS}
  			%\subjclass[2010]{Primary
  				35K57,  	%Reaction-diffusion equations
  				35Q92,  	%PDEs in connection with biology, chemistry and other natural sciences
  				65M22  	%Numerical solution of discretized equations for initial value and initial-boundary value problems involving PDEs
  				65M60,		%Finite element, Rayleigh-Ritz and Galerkin methods for initial value and initial-boundary value problems involving PDEs
  				92B05,  	%General biology and biomathematics
  			\end{AMS}

\section{Introduction}\label{sec-1}

	Coronary artery disease (CAD) is a condition where plaque builds up inside the coronary arteries, which are the blood vessels that supply oxygen-rich blood to the heart muscle. As the plaque accumulates, it can narrow or block the arteries, reducing blood flow to the heart and causing chest pain or discomfort, shortness of breath, fatigue, and other symptoms. CAD can also lead to more serious conditions, such as heart attack or heart failure.
	Treatments for CAD,  including angioplasty, vary depending on the severity and extent of the disease. In some cases, a stent may be placed during angioplasty. There are two main types of stents: bare-metal stents and drug-eluting stents (DESs). Bare-metal stents are made of metal and are effective at keeping the artery open, but they can sometimes cause re-narrowing of the artery, called restenosis. DESs are coated with medication that helps prevent re-narrowing and improve long-term outcomes.
\begin{figure}[!htb]
    \centering
    \includegraphics[width=0.5\textwidth]{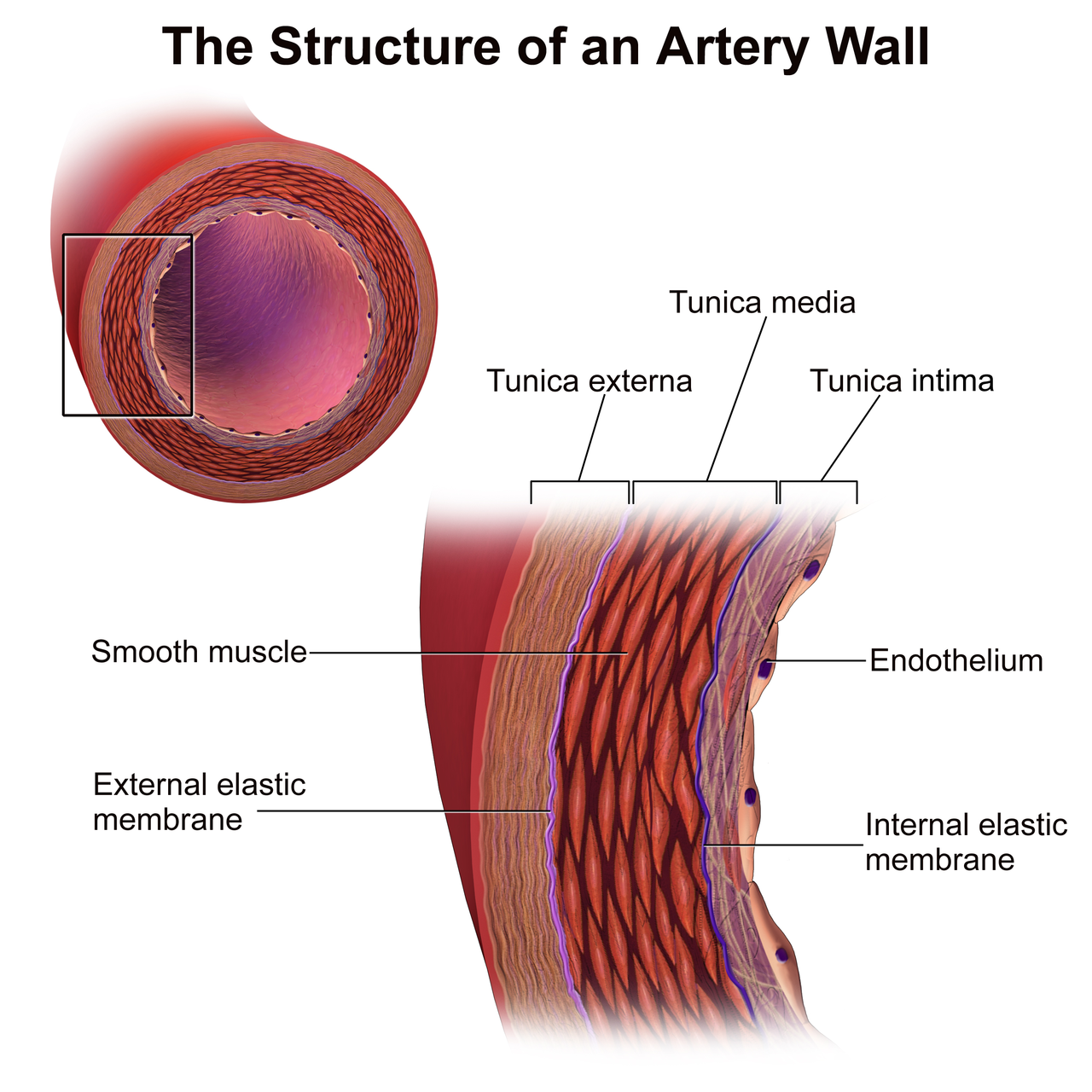}
         \caption[arterywall]{Sketch of arterial wall structure from \cite{blausen}.} %\footnotemark}
         \label{fig:ArteryWallStructure}
\end{figure}
	In order to model the drug delivery from the DES to and through the arterial walls, it is necessary to study the biological structures of the arteries. There are three layers that comprise the arterial wall (see \Cref{fig:ArteryWallStructure}), starting from the inside of the wall: intima, media, and adventitia. A thin layer of endothelial cells, called the endothelium, lines the inside of the intima. They are in contact with the blood and control the relaxation and contraction of the artery, as well as prevent the smooth muscle cells in the media from proliferating.
	The media is composed of smooth muscle cells, collagen, and elastic fibers that help to regulate blood pressure and flow. The smooth muscle cells are the targets for the drug delivery. 
	The adventitia is composed of connective tissue that supports the artery. It is filled with tiny blood vessels called vasa vasorum, which supplies blood to the adventitia and acts as a clearance mechanism for drugs released into the artery wall.

	Many multi-layer models have been proposed to study the pharmacokinetics in the arterial wall. Among the one-dimensional models, we first mention the model proposed in \cite{PontrelliDeMonte07}, which consists of a diffusion equation in the drug-coating region and a diffusion-advection-reaction equation in the arterial wall region. The coupling is achieved by applying interface conditions.  
	In \cite{McGinty11}, the authors further took the intracellular concentration into account, along with extending an early model to include the adventitia layer as well. In \cite{McGinty13}, the authors studied the 2-layer model from \cite{McGinty11} and provided an analytic solution in some special case. This two-layer model is the focus of this paper.
	
	 High-dimensional models have been studied as well. We refer the reader to  \cite{Balakrishnan, Borghi.et.al, ECPMM2020, Feenstra, Grassi.et.al, Horner, Hose.et.al, WSR, Zunino04} for more details. Here, we focus our attention on the model proposed in \cite{McGinty13} since it models the intracellular concentration separately.
	The reader is also referred to \cite{McGinty14} for a review of different models.

	The remainder of this paper is organized as follows. In Section \ref{sec-2}, we  introduce the one-dimensional model and state its weak formulation. In Section
	\ref{sec-3}, we present a complete PDE analysis for the model, which includes derivation of a priori energy estimates and the establishment of its well-posedness by using the Galerkin method with a compactness argument. In Section \ref{sec-4}, we present a complete finite-element numerical analysis for the PDE (partial differential equation) model, followed by the numerical results given in Section \ref{sec-5}. In Section \ref{sec-6},  we first introduce a generalized two-dimensional model and then sketch some  PDE and numerical analysis results for the proposed model. Finally, the paper is completed  with a few concluding remarks given in Section \ref{sec-7}. 
	%section is a brief summary with future directions. The eighth section is an appendix, gathering all parameters as well as statements cited throughout the work, and where necessary, their proofs.
	%
	%\footnotetext{Blausen.com staff (2014). ``Medical gallery of Blausen Medical 2014". WikiJournal of Medicine 1 (2). DOI:10.15347/wjm/2014.010. ISSN 2002-4436.}

%%%%%%%%%%%%%%%%%%%%%%%%%%%%%%%%%%%%%% 
\section{Mathematical model and its weak formulation}\label{sec-2}
     In this section, we first introduce the one-dimensional drug-release model from \cite{McGinty13} and then present its weak formulation.
	We note that several geometric simplifications were adopted when establishing  this 1-d model. First, the endothelium is usually severely damaged after the stent insertion; it is therefore omitted. Second, the intima, when devoid of the endothelium, has a structure that is similar to the media and will thus be absorbed into the media region in the model. Third, the adventitia is omitted in the model since research shows that it does not have a large effect on the drug concentration in the media region. See \Cref{1d} for a schematic diagram.
	
\begin{figure}\centering
\begin{tikzpicture}[scale=0.65]
\draw[fill=black!10,even odd rule] (0,0) circle (1.5 cm) circle (2 cm);
\node at(0,-0.5) {Lumen};
\node at(0,1.7) {Wall};
\draw [fill=black!30] (-0.3,1.3) rectangle (0.3,1.5);
\draw [fill=black!30] (-0.3,-1.3) rectangle (0.3,-1.5);
\draw [fill=black!30] (-1.5,-0.3) rectangle (-1.3,0.3);
\draw [fill=black!30] (1.5,-0.3) rectangle (1.3,0.3);
\node at (-3,1) {Stent};
\draw[->] (-2.5,1) to (0,1.4);
\draw[->] (-2.5,1) to (-1.4,0);
\draw [thick, red] (1.6,0) circle (0.6cm);
\draw [dotted, very thick, blue] (0,0) -- (1.3,0);
\draw [thick, blue] (1.3,0) -- (2,0);
\filldraw [black] (0,0) circle (1pt);
\filldraw [blue] (1.3,0) circle (1.5pt);
\filldraw [blue] (1.5,0) circle (1.5pt);
\filldraw [blue] (2,0) circle (1.5pt);
\draw[->,red,thick] (1.7,0.1) to[bend left] (4,0.2);
\draw [thick, blue] (4.5,0) -- (10,0);
\filldraw [blue] (4.5,0) circle (2pt);
\filldraw [blue] (5.5,0) circle (2pt);
\filldraw [blue] (10,0) circle (2pt);
\node at(4.5,-0.5) {$x=-l$};
\node at(5.5,0.5) {$x=0$};
\node at(10,-0.5) {$x=1$};
\end{tikzpicture}
\caption{1-d schematic diagram.}
\label{1d}
\end{figure}
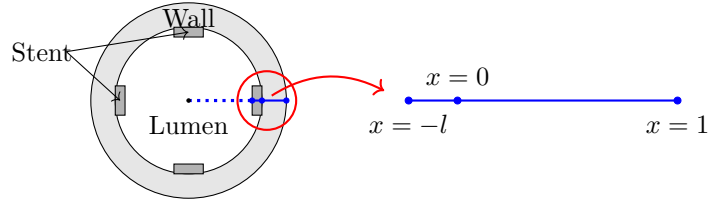

	Let $c$, $c_1$, and $c_2$ denote, respectively, the concentrations of the drug in the stent coating, in the extracellular matrix, and in the smooth muscle cells. The stent concentration is governed by the diffusion equation, the extracellular concentration by the diffusion-advection-reaction equation, and the intracellular concentration by a linear Ordinary Differential Equation (ODE). A no-flux boundary condition is imposed at the lumenal boundary ($x=-l$). The stent and wall concentrations are coupled through the continuity of mass flux, as well as the Kedem-Katchalsky equation at the interface ($x=0$). The system is then non-dimensionalized. We refer the reader to \cite{McGinty13} for a detailed explanation. 
	However, we note that the original model proposed in \cite{McGinty13} imposes a boundedness condition on the solution, whose main purpose is to help one to obtain an analytic solution, but this restriction may not be appropriate from the PDE point of view and, more importantly, it is difficult to approximate  numerically. Hence, we chose to replace this boundedness condition by imposing a no-flux boundary condition on the adventitial boundary ($x=1$)\textcolor{black}{, under the assumption that no drug escapes through the adventitial boundary. We note that this is an idealized situation. It is also common to impose the homogeneous Dirichlet boundary condition there, under the assumption that the drug concentration would be negligible at the far end of the arterial wall. Between the two idealized situations, we chose the former, with the understanding that the analysis of the system would not be affected aside from having slightly different solution spaces.}
	%The coupled system will be solved by partitioning the multi-layer domain. 
	
	Specifically, let $\Os:=(-l,0)$ and $\Om:=(0,1)$; our one-dimensional drug-release model is given as follows:  

	\begin{alignat}{3}
	   	\p_t c - \del c_{xx} &= 0, 	    &\qquad& x\in\Os,    		       &\,&t>0, \label{PDE}\\
	   	c_x &= 0, 			   	    &\qquad& x=-l,                    		&&t>0, \label{BC-l}\\
	   	c_x+  \tP c &=  \tP c_1,          &\qquad& x=0, 		     		&&t>0, \label{BC0} \\
	   	c &= 1, 				    &\qquad& x \in \overline{\Omega}^s,          &&t=0, \label{IC} \\[12pt]
	   	\phi \p_t c_1 - (c_1)_{xx} + \Pe(c_1)_x + \Da c_1 &= \frac{\Da}{K} c_2, 
	   						      &\qquad& x\in\Om, 		        &&t>0,  \label{PDEi}\\
	   	(c_1)_x - \Pe c_1 &= \del c_x, &\qquad& x=0,                  		&&t>0,  \label{BCi0}\\
	   	(c_1)_x &= 0, 			      &\quad& x=1,                    		&&t>0,  \label{BCi1}\\
	   	c_1 &= 0, 				      &\quad& x \in \overline{\Omega}^m,         &&t=0, \label{ICi} \\[12pt]
	   	(1-\phi) \p_t c_2 + \frac{\Da}{K} c_2 &= \Da c_1, 
	   						       &\quad& x\in\Om,   		       &&t>0, \label{ODEii}\\
	  	 c_2 &= 0, 			       &\quad& x \in \overline{\Omega}^m,       &&t=0,   \label{ICii}
	\end{alignat}
where $\partial_t$ denotes the partial derivative in time $t$ and the sub-index $x$ represents the partial derivative in the spatial variable $x$. \textcolor{black}{The parameters $\del, \tP, \Pe, \Da, K$ are positive real constants, while $\phi$ is a constant real number between 0 and 1.} Their specific values, as they appear in \cite{McGinty13}, are summarized in Appendix \ref{sec-8-1}.

Following the standard derivations, we can obtain the following weak formulation for the above coupled system.

	\begin{definition} \label{WeakSoln}
		$(c, c_1, c_2)$ is called a weak solution for the system given by \eqref{PDE}--\eqref{ICii}  if
		\begin{align*}
			c &\in \Linf(0,T;L^2(\Os)) \cap L^2(0,T;H^1(\Os)) \cap H^1(0,T;H^{-1}(\Os)),\\
			c_1 &\in \Linf(0,T;L^2(\Om)) \cap L^2(0,T;H^1(\Om)) \cap H^1(0,T;H^{-1}(\Om)), 
		%	c_2 &\in C^1(0,T;L^2(\Om)), 
		\end{align*}
		and $c_2 \in {\color{black} H^1(0,T;L^2(\Om)) }$ satisfy that $c(\cd,0) \equiv 1$ and $c_1(\cd,0) \equiv c_2(\cd,0) \equiv 0$, and, for some $T>0$, any $t\in (0,T]$ and any 
		$(v,w)\in H^1(\Os)\times H^1(\Om)$ such that
		\begin{alignat}{3} 
			\< \p_t c, v \>_{H^{-1}(\Os) \x H^1(\Os)} + \calA[c, v] &= \del \tP c_1(0,\cd) v(0,\cd), 
										&& % \forall \,v\in H^1(\Os),    &\,& t\in(0,T]
			\label{WeakSoln1} \\
			\phi \< \p_t c_{1}, w\>_{H^{-1}(\Om)\x H^1(\Om)} + \calB[c_1, w] &= \del \tP c(0,\cd) w(0,\cd) + \frac{\Da}{K} (c_2, w)_\Om, 
										&& %\forall \,w\in H^1(\Om),  &\,& t\in(0,T]
			\label{WeakSoln2}\\
			(1-\phi) \p_t c_{2} + \frac{\Da}{K} c_{2} &= \Da c_1, 
										&& %\quad & a.e.\,x\in\Om,       &\,& t\in(0,T]
			\label{WeakSoln3}  
		\end{alignat}
		where $(\cd,\cd)$ denotes the $L^2$-inner product on the respective domain, $\<\cd,\cd\>$ denotes the duality pairing, and
		\begin{align*}
			\calA[w, v] &:=  \del(w_x, v_x)_\Os + \del \tP  w(0,\cd)v(0,\cd),
			\\
			\calB[w, v] &:= \Pe(w_x, v)_\Om + \Da(w, v)_\Om + (w_x, v_x)_\Om + (\del \tP +\Pe)w(0,\cd)v(0,\cd)
		\end{align*}
		with  $\del, \tP, \Pe, \Da, K$ \textcolor{black}{being} positive constant parameters \textcolor{black}{and $\phi$ being a constant between 0 and 1}.
	\end{definition}

	For notation brevity, but without loss of clarity, throughout this paper, we may omit the explicit domain dependence in spatial norms. For example, $\|\cd\|_{L^2(\Os)}$ could be written as $\|\cd\|_{L^2}$.

%%%%%%%%%%%%%%%%%%%%%%%%%%%%%%%%%%%%%%      
\section{PDE analysis}\label{sec-3}
 
 The goals of this section are to prove the well-posedness for the coupled PDE system given by \eqref{PDE}--\eqref{ICii} and establish some properties for the weak solution, including the boundedness property. 
 To this end, we first derive some needed a priori estimates for weak solutions. Then, we prove the existence and uniqueness by using the Galerkin and energy methods. 

%%%
\subsection{A priori estimates} 
The main result of this subsection is summarized in the following theorem. 

	\begin{theorem} \label{Apriori}
		Let $(c,c_1,c_2)$ be a weak solution to the system given by \eqref{PDE}--\eqref{ICii} in the sense of Definition \ref{WeakSoln}. Then, the following holds:
		\begin{align} \label{Apriori1}
			&\|c\|^2_{\Linf(0,T;L^2(\Os))} + \|c_1\|^2_{\Linf(0,T;L^2(\Om))} + \|c_2\|^2_{\Linf(0,T;L^2(\Om))} \\
			&\qquad + \frac{2\del}{\gam}\|c_x\|^2_{L^2(0,T;L^2(\Os))} 
				+ \frac{2}{\gam} \|c_{1x}\|^2_{L^2(0,T;L^2(\Om))} \nonumber \\
			&\qquad\quad	+ \frac{\Pe}{\gam} \Big( \|c_1(0,\cd)\|^2_{L^2(0,T)} +  \|c_1(1,\cd)\|^2_{L^2(0,T)} \Big)  \leq 3l^2 e^{MT}, \nonumber \\
		 \label{Apriori2}
		&\|c(0,\cd)\|^2_{L^2(0,T)}  \leq  \frac{3\gam l^2 e^{MT}}{\Pe} + \frac{2l^2}{\del \tP }, \\
		 \label{Apriori3}
		&\|\p_t c\|_{L^2(0,T;H^{-1}(\Os))} + \|\p_t c_1\|_{L^2(0,T;H^{-1}(\Om))} +  \|\p_t c_2\|_{L^2(0,T;L^2(\Om))}  \\
				&\quad \leq C \Big( \|c\|_{L^2(0,T;H^1(\Os))} + \|c_1\|_{L^2(0,T;H^1(\Om))}+ \|c_2\|_{L^2(0,T;L^2(\Om))} \Big), \nonumber
		 \end{align}
		 where $C>0$ is a constant that is independent of $(c_1,c_2,c)$.
	\end{theorem}

	\begin{proof} 
		Setting $v:=c$ and $w:=c_1$ in Definition \ref{WeakSoln}, by Theorem  \ref{evans_more_calculus}, H\"older's \textcolor{black}{inequality}, and the fundamental theorem of calculus, we get
		\begin{equation} \label{estimate1}
			\frac{1}{2}\frac{d}{dt} \|c\|^2_{L^2} + \del \|c_x\|^2_{L^2} + \del \tP c^2(0,\cd) 
				\leq \del \tP  \( \eps_1 c^2(0,\cd) + \frac{1}{4\eps_1} c_1^2(0,\cd) \)	\quad \forall \eps_1>0,
			\end{equation}
		\begin{multline} \label{estimate2}
			\frac{\phi}{2}\frac{d}{dt}\|c_1\|^2_{L^2} + \Da\|c_1\|^2_{L^2} + \|c_{1x}\|^2_{L^2} + \frac{\Pe}{2} c_1^2(1,\cd) +  \(\del \tP +\frac{\Pe}{2}\) c_1^2(0,\cd) \\
				\leq  \del \tP  \( \eps_2 c^2(0,\cd) + \frac{1}{4\eps_2} c_1^2(0,\cd) \) + \frac{\eps_3 \Da}{K} \|c_2\|^2_{L^2} +\frac{1}{4\eps_3} \|c_1\|^2_{L^2}								\quad \forall \eps_2,\eps_3>0,
		\end{multline}
		\begin{equation} \label{estimate3}
			\frac{(1-\phi)}{2} \frac{d}{dt} \|c_2\|^2_{L^2} + \frac{\Da}{K} \|c_2\|^2_{L^2} 
				\leq \Da \eps_4 \|c_1\|^2_{L^2} + \frac{\Da}{4\eps_4} \|c_2\|^2_{L^2}		\quad \forall \eps_4>0.
		\end{equation} 

		In order to cancel out boundary terms on the right-hand side, we decided to choose $\eps_1=\eps_2 = \frac{1}{2}$ and $\eps_3 = \eps_4 = 1$ and let $\gam:=\frac{1}{2} \min\{\phi,1-\phi\}$. Combining (\ref{estimate1}) to (\ref{estimate3}) together, applying Lemma \ref{gronwall} with 
		\begin{align*}
		x(t) &= \frac{2}{\gam} \Big( \del \|c_x\|^2_{L^2} + \|c_{1x}\|^2_{L^2} + \frac{\Pe}{2} \Big( c_1^2(1,\cd)+c_1^2(0,\cd) \Big),\\
		y(t) &= \|c\|^2_{L^2} + \|c_1\|^2_{L^2} + \|c_2\|^2_{L^2}, \\
		z(t) &\equiv 0, \qquad C(t) \equiv  \frac{1+\Da}{2\gam} =: M,
		\end{align*}
		and taking the supremum over $[0,T]$ on both sides yields 
		\begin{multline*}
			\|c\|^2_{\Linf(0,T;L^2)}   +  \|c_1\|^2_{\Linf(0,T;L^2)}   +   \|c_2\|^2_{\Linf(0,T;L^2)}   +   \frac{2\del}{\gam}\|c_x\|^2_{L^2(0,T;L^2)}  \\
				+  \frac{2}{\gam} \|c_{1x}\|^2_{L^2(0,T;L^2)}   +  \frac{\Pe}{\gam} \Big( \|c_1(0,\cd)\|^2_{L^2(0,T)}   +  \|c_1(1,\cd)\|^2_{L^2(0,T)} \Big)  \leq  3 l^2 e^{MT}, 
		\end{multline*}
	which proves \eqref{Apriori1}.
	
		Next, to recover the boundary term $\|c(0,\cd)\|^2_{L^2(0,T)}$, notice that the above estimate, in particular, implies that $\|c_1(0,\cd)\|^2_{L^2(0,T)}$ is controlled from the above. Therefore, setting $\eps_1=\frac{1}{2}$ and integrating (\ref{estimate1}) over $(0,T)$ yields
		\begin{equation*}
			\|c(0,\cd)\|^2_{L^2(0,T)}  
				\leq  \|c_1(0,\cd)\|^2_{L^2(0,T)} + 2l^2(\del \tP )^{-1}
				\leq  3\gam l^2 e^{MT} \Pe^{-1} + 2l^2(\del \tP )^{-1},
		\end{equation*}
	which gives \eqref{Apriori2}.

		We have yet to estimate the functional norms $\|\p_t c\|_{H^{-1}(\Os)}$ and $\|\p_t c_1\|_{H^{-1}(\Om)}$. It suffices to show that all terms in the bilinear forms are bounded. By Lemma \ref{trace_1d}, we immediately have the following estimate:
		\[  \del \tP w(0,\cd) v(0,\cd)  \leq C \|w\|_{H^1} \|v\|_{H^1} .  \]
		Notice that the constant $C$ here depends on the spaces to which the functions $w$ and $v$ belong. If both $w,v\in H^1(\Os)$, then $C = 2l^{-1}+1 =\calO(l^{-1})$. If both $w,v\in H^1(\Om)$, then $C = 3 =\calO(1)$. If $w\in H^1(\Os)$ and $v\in H^1(\Om)$, or vice versa, then $C = (6l^{-1}+3)^{1/2} =\calO(l^{-1/2})$, which is the only case that explicitly appears in Definition \ref{WeakSoln}. However, the first two cases appear implicitly within the bilinear forms $\calA[\cd,\cd]$ and $\calB[\cd,\cd]$.
Consequently, by H\"older's \textcolor{black}{inequality}, we get
		\begin{align*}
			\calA[w,v]  
				&\leq \del \|w_x\|_{L^2} \|v_x\|_{L^2}   +   \del \tP  (2l^{-1}+1) \|w\|_{H^1} \|v\|_{H^1} \\
				&\leq  \big( \del + \del \tP  (2l^{-1}+1) \big)  \|w\|_{H^1} \|v\|_{H^1},  \\
			\calB[w,v] 
				&\leq \|w_x\|_{L^2}\|v_x\|_{L^2}  +  \Da\|w\|_{L^2}\|v\|_{L^2} 
				+  \Pe\|w_x\|_{L^2}\|v\|_{L^2} 
				+ 3(\del \tP  + \Pe) \|w\|_{H^1}\|v\|_{H^1} \\
				&\leq (1+4\Pe+3\del \tP ) \|w\|_{H^1}\|v\|_{H^1}. 
		\end{align*}
		Therefore,
		\begin{align*}
			\|\p_t c\|_{L^2(0,T;H^{-1})} &+ \|\p_t c_1\|_{L^2(0,T;H^{-1})} +  \|\p_t c_2\|_{L^2(0,T;H^{-1})}  \\ 
				&\leq C \Big( \|c\|_{L^2(0,T;H^1)} + \|c_1\|_{L^2(0,T;H^1)}+ \|c_2\|_{L^2(0,T;L^2)} \Big),
		\end{align*}
		where $C = C \big( l^{-1},\del, \del \tP, \Pe, \phi^{-1}, (1-\phi)^{-1}, \Da, K^{-1} \big)$ with linear dependence on each argument.
		This then verifies \eqref{Apriori3} and hence concludes the proof. 
	\end{proof}

	\begin{remark}
		{\color{black}The boundedness now becomes a property of the weak solution via the compact embedding $H^1 ((x_1,x_2))\hookrightarrow L^\infty ((x_1,x_2))$ for any $x_1<x_2$. This validates our modified model and the newly imposed no-flux boundary condition. }
	\end{remark}

%%%
\subsection{Well-posedness}\label{sec-3-2}

	Since the system given by \eqref{PDE}--\eqref{ICii} is linear, uniqueness would be an immediate corollary of a priori estimates because, if there are two weak solutions, their  difference must satisfy the conditions of the same equations but take zero initial conditions, which, in turn, implies that the difference is zero. Hence, we have the following theorem.

	 \begin{theorem}[Uniqueness] \label{Uniqueness}
		There exists at most one weak solution $(c,c_1,c_2)$ to the problem given by \eqref{PDE}--\eqref{ICii} in the sense of Definition \ref{WeakSoln}.
	\end{theorem}

To prove the existence, we adopt the Galerkin method with a compactness argument. To setup our Galerkin approximation, let {\color{black} $\mathcal T^s := \cup_{j=1}^{N} I^s_j$ and $\mathcal T^m := \cup_{j=1}^{N} I^m_j$}  be uniform meshes on $\Os$ and $\Om$ respectively. 
Let $\{\psi^s_j\}_{j=1}^{N+1}, \{\psi^m_j\}_{j=1}^{N+1}$ be the standard linear finite element nodal basis functions on $\mathcal T^s$ and $\mathcal T^m$, respectively, and define
	\begin{align*}
		&V^s_h := \text{span}\{\psi^s_j\}_{j=1}^{N+1} \subset H^1(\Os), \\
		&V^m_h := \text{span}\{\psi^m_j\}_{j=1}^{N+1} \subset H^1(\Om),
	\end{align*}
	where {\color{black} $h = l/N$ in $V^s_h$ and $h=1/N$ in $V^m_h$. Here, we abuse the notation to give $h$ multiple meanings  for the sake of notation brevity. }

	\begin{definition} \label{ApproxSoln}
		$(c_h,c_{1h}, c_{2h}): [0,T] \to V^s_h\times V^m_h\times V^m_h$ is called an \textit{approximate weak solution} to the system given by \eqref{PDE}--\eqref{ICii} 
		if the following holds for any $(v_h,w_h) \in V^s_h\times V^m_h$:
		\begin{align} 
			&\< \p_t c_h, v_h \>_{H^{-1}(\Os)\x H^1(\Os)} + \calA[c_h, v_h] = \del \tP c_{1h}(0,\cd) v_h(0,\cd),  \\
			&\phi \< \p_t c_{1h}, w_h \>_{H^{-1}(\Om)\x H^1(\Om)} + \calB[c_{1h},w_h] \\
			&\hskip 1.7in = \del \tP c_h(0,\cd) w_h(0,\cd) + \frac{\Da}{K} (c_{2h},w_h)_\Om, \nonumber \\
			&(1-\phi) \p_t c_{2h} + \frac{\Da}{K} c_{2h} = \Da c_{1h}, 
		\end{align}
		with the initial conditions $c_h(\cd,0) \equiv 1$ and $c_{1h}(\cd,0) \equiv c_{2h}(\cd,0) \equiv 0$.
	\end{definition}

	\begin{lemma} \label{ode_soln}
		For each $h>0$, there exists a unique approximate weak solution $(c_h,c_{1h},c_{2h})$ in the sense of Definition \ref{ApproxSoln}.
	\end{lemma}
 
	\begin{proof} 
	By the definition,  $c_h$, $c_{1h}$, $c_{2h}$ can be written as follows:
		\begin{equation*}
				  c_h(x,t) = \sum_{j=1}^{N+1} y_{0,j}(t) \psi^s_j(x), 
			\, c_{1h}(x,t) = \sum_{j=1}^{N+1} y_{1,j}(t) \psi^m_j(x), 
			\, c_{2h}(x,t) = \sum_{j=1}^{N+1} y_{2,j}(t) \psi^m_j(x). 
		\end{equation*}
Then, the equations in Definition \ref{ApproxSoln} can be rewritten as follows:
		\begin{align} 
			&\sum_{i=1}^{N+1} y'_{0,i}(t) (\psi^s_i,\psi^s_j)_\Os + y_{0,i}(t) \calA[\psi^s_i,\psi^s_j] 
				= \del \tP c_{1h}(0,\cd) \psi^s_j(0),   																\label{FEM} \\
			&\phi \sum_{i=1}^{N+1} y'_{1,i}(t) (\psi^m_i,\psi^m_j)_\Om + y_{1,i}(t)\calB[\psi^m_i,\psi^m_j] \label{FEMi} \\
			&\hskip 0.8in = \del \tP c_h(0,\cd) \psi^m_j(0) + \frac{\Da}{K} \sum_{i=1}^{N+1} y_{2,i}(t) (\psi^m_i,\psi^m_j)_\Om, 	\nonumber	 \\
			&(1-\phi) y'_{2,j}(t) \psi^m_j(x) + \frac{\Da}{K} y_{2,j}(t) \psi^m_j(x) = \Da  y_{1,j}(t) \psi^m_j(x)								\label{FEMii}
		 \end{align}
	 for each $j=1,\cdots,N+1$.
 
 		\Crefrange{FEM}{FEMii} can be rewritten as the following ODE system:
		\begin{equation} \label{odes}
			\bm D \,\bm y' (t) = \bm M \,\bm y(t), \qquad \bm y(0)=\begin{bmatrix}\vec{1}\\\vec{0}\\\vec{0}\end{bmatrix},
		\end{equation}
		where \begin{align*}
				  \bm y(t) &= \begin{bmatrix} \bm y_0(t) \\ \bm y_1(t) \\ \bm y_2(t) \end{bmatrix} , 
			\qquad  \bm D = \begin{bmatrix} 
							\Psi_s & & \\ 
							& \phi \Psi_m & \\ 
							& & (1-\phi)I 
						\end{bmatrix}, \\
					\\
			 \bm M &= 
			\left[ \begin{array}{c@{}c@{}c}
				-A 					& 	\left[\begin{array}{cc}
 		        									\vdots & \reflectbox{$\ddots$} \\
		         								\del \tP  & \cdots \\
 		 									\end{array}\right] 		& \mathbf{0} \\
 			 	\left[\begin{array}{cc}
		         		\cdots & \del \tP  \\
		         		\reflectbox{$\ddots$} & \vdots \\
		  		\end{array}\right] 		& -B 							& \frac{\Da}{K} \Psi_s \\
			\mathbf{0} 				& \Da I 						& -\frac{\Da}{K} I \\
			\end{array} \right], 
			\end{align*}
		and
		\[ 
		 [\Psi_s]_{ij} = (\psi^s_i,\psi^s_j)_\Os, \,\, [\Psi_m]_{ij} = (\psi^m_i,\psi^m_j)_\Om, \,\, 
			A_{ij} = \calA[\psi^s_i,\psi^s_j], \,\, B_{ij} = \calB[\psi^m_i,\psi^m_j]. 
		\]
		We note that, for $j=1,2,\cdots,N+1$, the following holds:
		\[ 
		y_{0,j}(t) = c_h(x_j,t), \quad
		 y_{1,j}(t) = c_{1h}(x_j,t), \quad
		y_{2,j}(t) = c_{2h}(x_j,t).
		\]
		Hence, the existence of a unique approximate weak solution is equivalent to 
		the existence of a unique solution to the above ODE system. Since $\Psi_s$ and $\Psi_m$ are tri-diagonal and strictly diagonally dominant, they are  invertible. Then, $ \bm{y}\,' (t) = \bm D^{-1} \bm M \,\bm{y}(t)$. Since $\bm D^{-1} \bm M$ is a constant matrix, by Lemma \ref{hartman}, the ODE system has a unique solution. Thus, there exists a unique approximate weak solution $(c_h,c_{1h},c_{2h})$.
	\end{proof} 

	\begin{theorem}[Existence] \label{Existence}
		There exists a weak solution $(c,c_1,c_2)$ to the problem given by \eqref{PDE}--\eqref{ICii} in the sense of Definition \ref{WeakSoln}.
	\end{theorem}

	\begin{proof} 
		We first notice that the approximate weak solution proved in Lemma \ref{ode_soln} satisfies the conditions of those estimates of Theorem \ref{Apriori}. Since $L^2(0,T;H^1(\Os))$ is a reflexive Banach space, $\Linf(0,T;L^2(\Os))$ is a separable normed linear space, and the sequence $\{c_h\}$ is uniformly bounded in $h$, then there exists a subsequence $\{c_{h_j}\}$ that converges weakly and weak* in $\Linf(0,T;L^2(\Os))$ and $\Linf(0,T;L^2)$, respectively (see Theorems \ref{friedman1} to \ref{friedman3}), that is, there exists $c\in L^2(0,T;H^1(\Os)) \cap \Linf(0,T;L^2(\Os))$ such that
		\begin{gather*}
			\int_0^T \< f, c_{h_j} \> dt \to \int_0^T \< f, c \> dt \qquad \forall\, f \in L^2(0,T;H^{-1}(\Os)), 
			\\
			\int_0^T \< c_{h_j}, g \> dt \to \int_0^T \< c, g \> dt \qquad \forall\, g \in \Linf(0,T;L^2(\Os)).
		\end{gather*}

		Moreover, since $H^1(0,T;H^{-1}(\Os))$ is a separable normed linear space and $\{c'_h\}$ is uniformly bounded in $h$, there exists a subsequence of  $\{c_{h_j}\}$ (not relabeling here for notation brevity) such that it converges in a weak* sense; namely, there exists $\zeta\in L^2(0,T;H^{-1}(\Os))$ such that 
			\[ 
			\int_0^T \< \p_t c_{h_j}, g \> dt \to \int_0^T \< \zeta, g \> dt \qquad \forall\, g\in L^2(0,T;H^{1}(\Os)), 
			\]
		We want to show that $\zeta$ is actually the weak time derivative of $c$. To this end, let $\eta \in C_0^\infty(0,T)$ and $v_h \in V^s_h$; then, 
			\[ 
			\int_0^T \< \p_t c_{h_j}, \eta v_h \> dt = - \int_0^T \< c_{h_j}, \eta' v_h \> dt \to - \int_0^T \< c, \eta' v_h \> dt \quad\mbox{as } j\to \infty.
			\]
		Therefore, $\zeta$ is the weak time-derivative of $c$ by definition. 
		Now, let $\eta \in C^\infty_0(0,T)$ and $v_h \in V^s_h$; by Definition \ref{ApproxSoln}, we have 
		\begin{equation*} 
			\< \p_t c_{h_j},\eta v_h \>_{H^{-1}(\Os)\x H^1(\Os)} + \calA[c_{h_j},\eta v_h] = \del \tP c_{1h_j}(0,\cd) \eta v_h(0,\cd) 
		\end{equation*}

		For each fixed $\eta$ and $v_h$, notice that $\,\calA[\cd,\eta v_h]$ defines a bounded linear functional on the space to which $\{c_{h_j}\}$ belongs. Therefore, setting $j\to \infty$, and by weak convergence, we get 
		\[ 
	\int_0^T \eta 				\calA[c_{h_j}, v_h] dt =\int_0^T \calA[c_{h_j}, \eta v_h] dt \to \int_0^T \calA[c, \eta v_h] dt = \int_0^T \eta  \calA[c, v_h]  dt. 
		\]

		Similarly,
		\[ 
		\del \tP \int_0^T \eta c_{h_j}(0,\cd) v_h(0,\cd) dt \to \del \tP \int_0^T \eta c(0,\cd) v_h(0,\cd) dt \quad\mbox{as } j\to \infty.
		\]

Since $\eta\in C^\infty_0(0,T)$ is arbitrary, then the above equations infer that
		\begin{equation*} 
			\< \p_t c, v_h \>_{(H^{-1}\x H^1)(\Os)} + \calA[c, v_h] = \del \tP c_1(0,\cd) v_h(0,\cd) \qquad \forall\, v_h\in V^s_h,
		\end{equation*}
which, with the denseness of $V^s_h$ in $H^1(\Os)$, implies that 
		\begin{equation*} 
			\< \p_t c, v \>_{(H^{-1}\x H^1)(\Os)} + \calA[c, v] = \del \tP c_1(0,\cd) v(0,\cd) \qquad \forall\, v \in H^1(\Os).
		\end{equation*}
	Hence, \eqref{WeakSoln1} holds. 

Using exactly the same argument we can show that $c_1$ satisfies \eqref{WeakSoln2}.

	It remains to be shown that $c_2$ satisfies \eqref{WeakSoln3}. To this end, let $\eta \in C_0^\infty (0,T)$ and $w_h\in V^m_h$; then, it follows that
			\[ 
			\int_0^T \< \eta w_h, c'_{2h}\> dt + \int_0^T \Bigl\< \frac{\Da}{K}c_{2h}, \eta w_h \Bigr\> dt = \int_0^T \< \eta w_h, \Da c_{1h}\> dt. 
			\]
	Using a previous derivation, we can show that
			\[ 
			\int_0^T \< \eta w_h, \Da c_{1h_j}\> dt \to \int_0^T \< \eta w_h, \Da c_1\> dt\quad \mbox{as } j\to \infty.
			\]

		In addition, since $\{c_{2h_j}\}$ is uniformly (in $h$) bounded in $\Linf(0,T;L^2)$, which is a separable normed linear space, and $\{c'_{2h_j}\}$ is uniformly bounded in $L^2(0,T;L^2)$, which is a reflexive Banach space, then there exists a subsequence of $\{c_{2h_j}\}$ (not relabeling for notation brevity) such that, as $j\to\infty$,
			\[ 
			c_{2h_j} \weakto c_2 \in \Linf(0,T;L^2(\Om)), \quad \p_t c_{2h_j} \weakto^* \theta \in L^2(0,T;L^2(\Om)). 
			\]
Again, with the help of integration by parts and the definition, we can show that $\theta = \p_t c_2$. Therefore, as $j\to \infty$,
		\begin{gather*}
			\int_0^T \Bigl\< \frac{\Da}{K} c_{2h_j}, \eta w_h \Bigr\> dt \to \int_0^T \Bigl\< \frac{\Da}{K} c_2, \eta w_h \Bigr\> dt, \\
			\int_0^T \<  \eta w_h, \p_t c_{2h_j}\> dt \to \int_0^T \< \eta w_h, \p_t c_2\> dt. 
		\end{gather*}
Consequently, 
			\[ 
			\int_0^T \<  \eta w_h, \p_t c_2\> dt + \int_0^T \Bigl\< \frac{\Da}{K}c_{2}, \eta w_h \Bigr\> dt = \int_0^T \< \eta w_h, \Da c_1\> dt. 
			\]
Since $\eta\in C_0^\infty(0,T)$ and $w_h \in V_h^m$ are arbitrary and $V^m_h$ is dense in $H^1(\Omega^m)$, then
			\[ 
			(1-\phi) \p_t c_2 + \frac{\Da}{K}c_{2} =\Da c_1 \qquad \mbox{in } L^2(0,T; L^2(\Omega^m)). 
			\]
			
Thus,  $(c,c_1,c_2)$ is a weak solution to the problem given by \eqref{PDE}--\eqref{ICii} by Definition \ref{WeakSoln}.  The proof is complete. 
	\end{proof}

		\begin{corollary}[Convergence] \label{Convergence}
			The finite-element approximate weak solution $(c_h,c_{1h},c_{2h})$ converges to the unique PDE solution $(c, c_1, c_2)$. 
		\end{corollary}

\begin{proof}
From the proof of Theorem \ref{Existence}, we conclude that every convergent subsequence of the finite-element approximate weak solution $(c_h,c_{1h}, c_{2h})$ converges to a PDE weak solution $(c, c_1, c_2)$. Since the PDE weak solution is unique, the whole sequence $(c_h,c_{1h},c_{2h})$ must converge to the PDE solution $(c, c_1, c_2)$.
\end{proof}

%%%%%%%%%%%%%%%%%%%%%%%%%%%%%%%%%%%%%% 
\section{Error estimates and formulation of fully discrete scheme}\label{sec-4}
In the last section, we constructed a semi-discrete finite-element Galerkin 
approximation to the problem given by \eqref{PDE}--\eqref{ICii} and proved its convergence (see Corollary \ref{Convergence}) as a byproduct of the proof 
of the existence theorem. The primary goals of this section are to derive optimal rates of convergence in powers of $h$ (i.e., error estimates) for the finite-element solution, and to formulate a practical fully discrete scheme which will be used in the subsequent section for numerical simulations and to numerically verify the sharpness of the proved  convergence rates.

%%%
\subsection{Error estimates for semi-discrete finite element method}\label{sec-4-1}
We recall that PDE and finite-element approximate solutions were respectively defined in Definitions \ref{WeakSoln} and \ref{ApproxSoln}. To derive 
the error estimates, we first need to obtain the error equations; to this end, 
subtracting the equations in Definition \ref{ApproxSoln} from their corresponding equations in Definition \ref{WeakSoln} (with the same test functions $v_h \in V^s_h$ and $ w_h \in V^m_h$), we get
	\begin{align} 
		&\< \p_t e_h, v_h \>_{H^{-1}(\Os)\x H^1(\Os)} + \calA[e_h, v_h] 
			= \del \tP e_{1h} (0,\cd) v_h(0,\cd),  
		\label{Err} \\
		&\phi \< \p_t e_{1h}, w_h \>_{H^{-1}(\Om)\x H^1(\Om)} + \calB[e_{1h}, w_h] \label{Erri}\\
			&\hskip 0.8in = \del \tP e_h(0,\cd) w_h(0,\cd) + \frac{\Da}{K} (e_{2h}, w_h)_\Om,  \nonumber		 \\
		&(1-\phi)\p_t e_{2h} + \frac{\Da}{K} e_{2h}
			= \Da e_{1h}, \qquad a.e. \,x\in\Om,
		\label{Errii} 
	\end{align}
where $e_h := c-c_h, e_{1h} := c_1-c_{1h}$, and $e_{2h} := c_2-c_{2h}$.

Let $\calRh^A: H^1(\Os) \to V^s_h$ be the elliptic projection defined by
\[
\calA[ c - \calRh^A c, v_h]=0\qquad \forall v_h \in V^s_h,
\]
and $\calRh^B: H^1(\Os) \to V^m_h$ be another elliptic projection defined by
$$
\calB[ c_1 - \calRh^B c_1, w_h]=0\qquad \forall w_h\in V^m_h.
$$
These projection operators are well-defined because each bilinear form is coercive and continuous. Further, let $\calPh$ be the $L^2$-projection onto $V^m_h$ defined by
$$
(c_2 - \calPh c_2, w_h)=0\qquad \forall w_h\in V^m_h.
$$

We also introduce the following error decompositions:
	\begin{align*}
		e_h &= \underbrace{(c - \calRh^A c)}_{=:\rho_h} + \underbrace{(\calRh^A c - c_h)}_{=:\xi_h} ,
		\\
		e_{1h} &= \underbrace{(c_1 - \calRh^B c_1)}_{=:\rho_{1h}} + \underbrace{( \calRh^B c_1 - c_{1h})}_{=:\xi_{1h}} , 
		\\
		e_{2h} &= \underbrace{(c_2 - \calPh c_2)}_{=:\rho_{2h}} + \underbrace{( \calPh c_2 - c_{2h})}_{=:\xi_{2h}} .
	\end{align*}
	It is well known (see \cite{BrennerScott}) that 
\begin{align}
	\|\rho_h\|_{L^2} + h \| (\rho_h)_x\|_{L^2} &\lesssim h^2 \|c\|_{H^2}, 
	\label{elliptic_estimate1}\\
	\|\rho_{1h}\|_{L^2} + h \|(\rho_{1h})_x\|_{L^2}  &\lesssim h^2 \|c_1\|_{H^2}, 
	\label{elliptic_estimate2}\\
	 \|\rho_{2h}\|_{L^2}  &\lesssim h^2 \|c_2\|_{H^2}. 
	 \label{elliptic_estimate3}
\end{align}
Using the error decompositions we can rewrite the error equations \eqref{Err}--\eqref{Errii} as follows:
	\begin{align} 
	&\< \p_t \xi_h, v_h \>_{H^{-1}(\Os)\x H^1(\Os)} + \calA[\xi_h, v_h] 
	- \del \tP \xi_{1h} (0,\cd) v_h(0,\cd)\label{Err_2}  \\
	&\qquad = -\< \p_t \rho_h, v_h \>_{H^{-1}(\Os)\x H^1(\Os)} - \calA[\rho_h, v_h] 	+ \del \tP \rho_{1h} (0,\cd) v_h(0,\cd), \nonumber
	\\
	&\phi \< \p_t \xi_{1h}, w_h \>_{H^{-1}(\Om)\x H^1(\Om)} + \calB[\xi_{1h}, w_h]-\del \tP \xi_h(0,\cd) w_h(0,\cd) \label{Erri_2}\\
	&\qquad  - \frac{\Da}{K} (\xi_{2h}, w_h)_\Om 
	= -\phi \< \p_t \rho_{1h}, w_h \>_{H^{-1}(\Om)\x H^1(\Om)} - \calB[\rho_{1h}, w_h] \nonumber \\
	&\hskip 1.6in +\del \tP \rho_h(0,\cd) w_h(0,\cd) 
	+ \frac{\Da}{K} (\rho_{2h}, w_h)_\Om, \nonumber		 \\
	&(1-\phi)\p_t \xi_{2h} + \frac{\Da}{K} \xi_{2h}
	- \Da \xi_{1h} =(\phi-1)\p_t \rho_{2h} - \frac{\Da}{K} \rho_{2h}
	+ \Da \rho_{1h}. %,\quad a.e. \,x\in\Om,
	\label{Errii_2} 
\end{align}

To derive the desired error estimates, our task now is to control $\xi$ terms via the $\rho$ terms by using the above error equations. 

	\begin{theorem}[Error estimates] \label{FEMerr}
	The following error estimates hold:
	{\color{black} 
	\begin{align} \label{error_estimate1}
		\|e_h\|_{\Linf(0,T;L^2(\Os))} + \|e_{1h}\|_{\Linf(0,T;L^2(\Om))} + \|e_{2h}\|_{\Linf(0,T;L^2(\Om))} &\les h^2,\\
		\|(e_h)_x\|_{L^2(0,T;L^2(\Os))}   
		+ \|(e_{1h})_x\|_{L^2(0,T;L^2(\Om))}  +  \|e_h(0,\cd) \|_{L^2(0,T)} & \label{error_estimate2} \\
		+ \|e_{1h}(0,\cd)\|_{L^2(0,T)} + \|e_{1h}(1,\cd)\|_{L^2(0,T)} 	&\les  h. \nonumber
	\end{align}
}
	\end{theorem}

	\begin{proof}
		Similar to the a-priori estimates, setting $\xi_h, \xi_{1h}$, and $\xi_{2h}$ to be the test functions, we obtain the following inequalities:
		\begin{align}\label{eh}
		&	\frac{1}{2}\frac{d}{dt} \|\xi_h\|^2_{L^2} + \del \|(\xi_h)_x\|^2_{L^2} + \del \tP  \Big( \xi_h^2(0,\cd) \Big) \\
			&\qquad 	\leq \frac{1}{2} \|\xi_h\|^2_{L^2}  +  \del \tP  (\eps_1+\eps_2) \xi_h^2(0,\cd) + \frac{\del \tP }{4\eps_1}\xi_{1h}^2(0,\cd) + E_h,  \nonumber \\
		&	\frac{\phi}{2}\frac{d}{dt}\|\xi_{1h}\|^2_{L^2} + \Da\|\xi_{1h}\|^2_{L^2} + \|(\xi_{1h})_x\|^2_{L^2} + \frac{\Pe}{2} \xi^2_{1h}(1,\cd) + \(\del \tP  + \frac{\Pe}{2} \) \xi_{1h}^2(0,\cd) \label{e1h}  \\
		&\qquad \leq  
				\Big(\frac{\Da}{K} + \frac{\phi}{2} \Big) \|\xi_{1h}\|^2_{L^2}  
				+ \frac{\Da}{2K}\|\xi_{2h}\|^2_{L^2} +\Big( \frac{\del \tP }{4\eps_4} + \frac{\del \tP }{4\eps_5} \Big) \xi_{1h}^2(0,\cd) \nonumber\\
		&\qquad \quad	+ \del \tP  \eps_4 \xi^2_h(0,\cd)  +  E_{1h}, \nonumber \\
		&	\frac{1-\phi}{2}\frac{d}{dt} \|\xi_{2h}\|^2_{L^2} + \frac{\Da}{K} \|\xi_{2h}\|^2_{L^2} 	\label{e2h}\\ 
		&\qquad \leq \( \Da + \frac{1-\phi}{2} + \frac{\Da}{2K} \) \|\xi_{2h}\|^2_{L^2} + \frac{\Da}{2} \|\xi_{1h}\|^2_{L^2} + E_{2h}, \nonumber
		\end{align}
		where 
		\begin{align*}
		&E_h = \frac{\del \tP }{4\eps_2} \rho_{1h}^2(0,\cd) + \frac{1}{2} \|\p_t \rho_h\|^2_{L^2}, \\
		&E_{1h} = \eps_5 \del \tP   \rho_h^2(0,\cd) + \frac{\Da}{2K} \|\rho_{2h}\|^2_{L^2} + \frac{\phi}{2}\|\p_t \rho_{1h}\|^2_{L^2}, \\
		&E_{2h} = \frac{\Da}{2} \|\rho_{1h}\|^2_{L^2} + \frac{1-\phi}{2} \|\p_t \rho_{2h}\|^2_{L^2} + \frac{\Da}{2K} \|\rho_{2h}\|^2_{L^2}.
		\end{align*}

		Choose appropriate values of $\eps_2$ and $\eps_5$ such that $$ \tP  \eps_2 \xi_h^2(0,\cd) \leq \frac{\del}{2} \|\xi_h\|^2_{H^1(\Os)}  \quad\mbox{and}\quad
		 \frac{\del \tP }{4\eps_5} \xi_{1h}^2(0,\cd) \leq \frac{1}{2}\|\xi_{1h}\|^2_{H^1(\Om)}.$$
		Furthermore, choosing 
			$\eps_1=\eps_4=\frac{1}{2}$, as well as 
			$$\gam := \min\Bigl\{\frac{\phi}{2},\frac{1-\phi}{2}\Bigr\}, \quad
			\theta:=\max\Bigl\{\frac{1+\del}{2},\frac{\Da}{K} + \frac{\phi}{2} + \frac{\Da}{2},\Da + \frac{1-\phi}{2} \Bigr\},$$ 
		and adding (\ref{eh}) to (\ref{e2h}), yields
			\begin{multline*}
			\frac{d}{dt} \Big( \|\xi_h\|^2_{L^2} + \|\xi_{1h}\|^2_{L^2} + \|\xi_{2h}\|^2_{L^2} \Big) + \frac{2\del}{\gam} \|(\xi_h)_x\|^2_{L^2} + \frac{2}{\gam} \|(\xi_{1h})_x\|^2_{L^2}
			\\ \hspace{1cm} + \frac{\Pe}{\gam} \xi_{1h}^2(0,\cd) + \frac{\Pe}{\gam} \xi^2_{1h}(1,\cd) 
			  	\leq \frac{2\theta}{\gam} \Big( \|\xi_h\|^2_{L^2} + \|\xi_{1h}\|^2_{L^2} + \|\xi_{2h}\|^2_{L^2} \Big) +2E(t),
		\end{multline*}
	where $E:=E_h + E_{1h} + E_{2h}$.

	By using Gronwall's inequality given in Lemma \ref{gronwall} and taking the supremum over $(0,T)$, we get
		\begin{align*}
		&	\|\xi_h\|^2_{\Linf(0,T;L^2)} + \|\xi_{1h}\|^2_{\Linf(0,T;L^2)} + \|\xi_{2h}\|^2_{\Linf(0,T;L^2)} + \frac{2\del}{\gam} \|(\xi_h)_x\|^2_{L^2(0,T;L^2)} \\
		&\qquad	+ \frac{2}{\gam} \|(\xi_{1h})_x\|^2_{L^2(0,T;L^2)}   
			+ \frac{\Pe}{\gam} \|\xi_{1h}(0,\cd)\|^2_{L^2(0,T)} + \frac{\Pe}{\gam} \|\xi_{1h}(1,\cd)\|^2_{L^2(0,T)} \\
		&\qquad\qquad 	  	\leq 2e^{2\theta T/\gam} \|\mathcal{E}(t)\|_{L^2(0,T)} \leq C(e^{2\theta T/\gam}) \,\, h^4 .
		\end{align*}
		In particular,
			\[ \|\xi_{1h}(0,\cd)\|^2_{L^2(0,T)} \les h^4. \]

		By applying (\ref{eh}) with the above choices of $\eps_1$ and $\eps_2$, we have
		\begin{equation*}
			\frac{1}{2}\frac{d}{dt} \|\xi_h\|^2_{L^2} + \frac{\del}{2} \|(\xi_h)_x\|^2_{L^2} + \frac{\del \tP }{2} \xi_h^2(0,\cd) 
				\leq \frac{1}{2} \|\xi_h\|^2_{L^2} + \frac{\del \tP }{2}\xi_{1h}^2(0,\cd) + E_h.
		\end{equation*}
		Integrating over $(0,T)$ in $t$ yields
		\begin{align*}
			\|\xi_h(\cd,T)\|^2_{L^2} &+ \del \|(\xi_h)_x\|^2_{L^2(0,T;L^2)} + \del \tP  \|\xi_h(0,\cd) \|^2_{L^2(0,T)} \\
							    &\leq \del \tP  \|\xi_{1h}(0,\cd)\|^2_{L^2(0,T)} + 2 \|E_h\|_{L^2(0,T)} + \|\xi_h(\cd,0)\|^2_{L^2}  \les  h^4 .
		\end{align*}
		In particular,
			\[ \|\xi_h(0,\cd) \|^2_{L^2(0,T)} \les Ch^4 . \]

In summary, we have shown that 
		\begin{align*}
		&	\|\xi_h\|^2_{\Linf(0,T;L^2)} + \|\xi_{1h}\|^2_{\Linf(0,T;L^2)} + \|\xi_{2h}\|^2_{\Linf(0,T;L^2)} + \|(\xi_h)_x\|^2_{L^2(0,T;L^2)}  \\ 
		&\qquad	+ \|(\xi_{1h})_x\|^2_{L^2(0,T;L^2)}   
		 	+ \|\xi_h(0,\cd) \|^2_{L^2(0,T)} + \|\xi_{1h}(0,\cd)\|^2_{L^2(0,T)}  \\
		 &\qquad\qquad	+ \|\xi_{1h}(1,\cd)\|^2_{L^2(0,T)} \les  h^4,
		\end{align*}
which, combined with \eqref{elliptic_estimate1}--\eqref{elliptic_estimate3} and an application of the triangle inequality, concludes the proof. 
	\end{proof}

{\color{black}
\begin{remark}
	(a) Both estimates \eqref{error_estimate1} and \eqref{error_estimate2} are optimal compared to the linear finite-element interpolation errors. 
	
	(b) The $L^2$-norms of $(\xi_h)_x$ and $(\xi_{1h})_x$ exhibit a superconvergence property. 
	
	(c) If $r$th ($r>1$)-order finite-element space is used in place of the linear finite-element space and we assume that the solution $(c,c_1,c_2)$ is sufficiently regular, then it can be proved that the rates of convergence in \eqref{error_estimate1} and \eqref{error_estimate2} will be improved to $O(h^{r+1})$ and $O(h^r)$, respectively. 
\end{remark}
}

%%%
\subsection{Formulation of fully discrete schemes}\label{sec-4-2}
To get a computable fully discrete method, we need to discretize (\ref{FEM}) to (\ref{FEMii}) in time by using any time-stepping scheme, such as the Euler, implicit Euler, Runge-Kutta,
backward differentiation formula (BDF), and Crank-Nicolson methods. Below, we use the simplest Euler method to demonstrate the procedure. For each $j$, the Euler method is given by
	\begin{align*}
		&\sum_{i=1}^{N+1} \frac{y_{0,i}^{k+1}-y_{0,i}^k}{\Del t} (\psi^s_i,\psi^s_j)_\Os + y_{0,i}^k \calA[\psi^s_i,\psi^s_j] = \del \tP c_{1h}(0,\cd) \psi^s_j(0),   
		\\
		&\phi \sum_{i=1}^{N+1} \frac{y_{1,i}^{k+1}-y_{1,i}^k}{\Del t} (\psi^m_i,\psi^m_j)_\Om + y_{1,i}^k \calB[\psi^m_i,\psi^m_j] \\
		&\hskip 1.5in = \del \tP c_h(0,\cd) \psi^m_j(0) + \frac{\Da}{K} \sum_{i=1}^{N+1} y_{2,i}^k (\psi^m_i,\psi^m_j)_\Om,  
		\\
		&(1-\phi) \frac{y_{2,j}^{k+1}-y_{2,j}^k}{\Del t} \psi^m_j(x) + \frac{\Da}{K} y_{2,j}^k \psi^m_j(x) = \Da  y_{1,j}^k \psi^m_j(x).
	\end{align*}
	They can be rewritten in matrix-vector form, as follows:
	\begin{align} \label{linear_system1}
		\Psi_s \bm y_0^{k+1} &= (\Psi_s - \Del t A) \bm y_0^k + \Del t \bm f_0^k, 
		\\
		\Psi_m \bm y_1^{k+1} &= \( \Psi_m - \frac{\Del t}{\phi} B \) \bm y_1^k + \frac{\Del t \Da}{\phi K} \Psi_m \bm y_2^k + \frac{\Del t}{\phi} \bm f_1^k , 
		\label{linear_system2}\\
		\bm y_2^{k+1} &= \( 1-\frac{\Del t \Da}{(1-\phi)K} \) \bm y_2^k + \frac{\Del t \Da}{1-\phi} \bm y_1^k  , \label{linear_system3}
	\end{align}
	where $\bm f_0^k = [0; ...; 0; y^k_{1,1}]$ and $\bm f_1^k = [y^k_{0,N+1}; 0; ...; 0]$.

It is well known that the Euler method results in an error of order $\calO(\Del t)$; therefore,  it can be shown that the fully discrete error is of order $\calO(\Del t + h^2)$, provided that the Courant-Friedrichs-Lewy (CFL) condition $\Del t < \min\{\frac{h^2}{2\del}, \frac{\phi h^2}{2}\}$ holds.

%%%%%%%%%%%%%%%%%%%%%%%%%%%%%%%%%%%%%% 
\section{Numerical simulations}\label{sec-5}

	In this section we present some numerical simulation results, which were computed by using Matlab R2022a. Since we aimed to solve a coupled system, a decoupling strategy was needed. In addition, due to the biological nature of the system, we propose a multi-rate time-stepping procedure to solve the system in order to save computation time. Several comparisons are made in \Crefrange{sec-5-1}{sec-5-2} among different schemes, different time-stepping strategies, and different decoupling strategies. Numerical results of the simulations are presented in \Cref{sec-5-3}.
	
	The following notations are adopted in this section. Let $h$ denote the mesh size and $\Del t$ the time step size. When those values differ in the two domains, we denote them as $h_\Os$, $\Del t_\Os$ and $h_\Om$, $\Del t_\Om$, respectively. Furthermore, let $N_0 := l/h_\Os, N := 1/h_\Om, n_0 := T/\Del t_\Os$, and $n := T/\Del t_\Om$.

%%%
\subsection{Comparison between two decoupling strategies}\label{sec-5-1}

We propose two decoupling strategies for solving the system given by \eqref{linear_system1}--\eqref{linear_system3}. The first strategy is parallelizable, updating $c$ and $c_2$ simultaneously, followed by updating $c_1$, as shown in \Cref{AlgI}. The second strategy is a sequential update, which is shown in \Cref{AlgII}.

	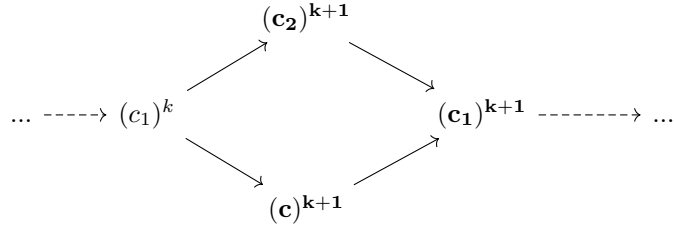
\begin{figure}%[h]
		\begin{center}
			\begin{tikzcd}
				&& \mathbf{(c_2)^{k+1}} \ar[dr] 
				&
				&[1.5em] \\
				...\ar[r,dashed] & (c_1)^k \ar[ur] \ar[dr]
				&
				& \mathbf{(c_1)^{k+1}} \ar[r,dashed]
				& ... \\
				&&\mathbf{(c)^{k+1}} \ar[ru]
				&
				&
			\end{tikzcd}
		\caption{Algorithm I based on decoupling strategy \#1.}
		\label{AlgI}
		\end{center}
	\end{figure}

	\begin{figure}%[h]
		\begin{center}
			\begin{tikzcd}
				...\ar[r,dashed] & \Big( (c),(c_1),(c_2) \Big)^k \ar[r] & \mathbf{(c_2)^{k+1}} \ar[r] & \mathbf{(c_1)^{k+1}} \ar[r] & \mathbf{(c)^{k+1}} \ar[r,dashed] &...
 			\end{tikzcd}
		\caption{Algorithm II based on decoupling strategy \#2.}
			\label{AlgII}
		\end{center}
	\end{figure}
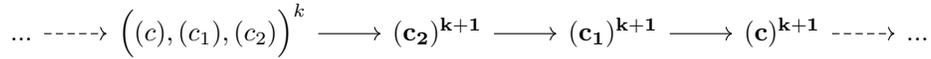

% CPU
	Test cases were run for $T_{end} = 1,10,100,200$, also applying different stepping strategies. 
	In some cases, Algorithm I ran faster than Algorithm II, and, in some cases, it was vice versa. Furthermore, the difference in total CPU time was within 5\% between the two algorithms. 
%  Accuracy
	As for the accuracy comparison, we computed the numerical ``true" solutions first, using $T=1, \widetilde N_0= 1000,  \widetilde N = 500,  \widetilde n = 2.5816\times 10^{6}.$ Then, the approximate solutions were computed for $T=1, N_0 = 50, N = 25, n = 6454.$
	The relative errors in $c$ and $c_2$ were both under $10^{-3}$; the relative error in $c_1$ was around $10^{-2}$ with Algorithm I, and it was about 10\% more accurate than Algorithm II. This can be explained by the fact that $c_1$ is updated with the most recent coupling values in Algorithm I.

%%%
\subsection{Comparison between two time-stepping strategies}\label{sec-5-2}

	Since each subdomain possesses distinct biological/physical properties, it is natural to use different mesh sizes for different subdomains. In addition,   the time step sizes can also be distinct in different subdomains; one scenario was that one time step size was taken as a constant multiple of the other.
	Recall that the Robin boundary condition of $\Os$ at $x=0$ states that $c_x + \tP c=\tP c_1$, where $\tilde{P} = 45000$. Therefore, it is natural to use a fine spatial mesh in $\Os$. This does not have much effect on the mesh size despite the CFL condition since the diffusion coefficient was extremely small, at $\del \sim 10^{-7}$. Therefore, the explicit time-stepping in fact demonstrates no disadvantage in this regard.

	We restricted ourselves to Algorithm I and compared the performance of the naive and the multi-rate time-stepping strategies. 
	When $N_0=2N$, all three errors decreased to around 10\%, and, in the case of $c_2$, the errors decreased to around 1\% of their naive counterparts, without increasing CPU time. \Cref{tab:steppingACCU} shows the relative errors of the three equations respectively. 
	\begin{table}%[h]
	\centering
	\caption{Stepping strategies: accuracy comparison.}
		\begin{tabular}{ c c c c } 
		   \hline
		   & Rel. err. in $c$ & Rel. err. in $c_1$ & Rel. err. in $c_2$ \\
		   \hline
		$N_0=N$ & 2.75e-3 & 0.1212 & 7.12e-4 \\
		$N_0=2N$ & 3.52e-4 & 8.42e-3 & 1.59e-6 \\
		\hline
		\end{tabular}
	
	\label{tab:steppingACCU}
	\end{table}
	
	It is worth pointing out that increasing the ratio $N_0/N$ further would increase CPU time without seeing any improvement in the accuracy. Furthermore, choosing $N_0<N$ not only increased CPU time, it also resulted in larger error.

%%%
\subsection{Simulation results}\label{sec-5-3}

	We computed the concentrations $c, c_1, c_2$ with $T$ taken as 10 minutes, 30 minutes, 1 hour, 6 hours, and 24 hours after the stent insertion, using Algorithm I. Computed results are summarized in \Cref{fig:snapsA} and \Cref{fig:snapsB}.
	It can be observed that the diffusion within the stent is slow. This is in line with the extremely small diffusion coefficient, i.e., $\del \ll1$. However, at the interface, the advection of the drug into the media region is initially extremely fast since it is proportional to the concentration difference, which led to the steep drop near $x=0$. 
	As for the media concentrations, $c_1$ increased for a period of time before eventually dropping due to the absorption of the drug into the  muscle cells, where a steady increase in drug concentration is demonstrated. 
	In addition, \Cref{fig:interface1} shows the interface concentration $c_1(0,t)$ over the first 24 hours after stent insertion, reproducing the profile shape from \cite{McGinty13}, which was obtained via an analytic method. 
	\begin{figure}[h]
	  \centering
 		 \includegraphics[width=0.8\linewidth]{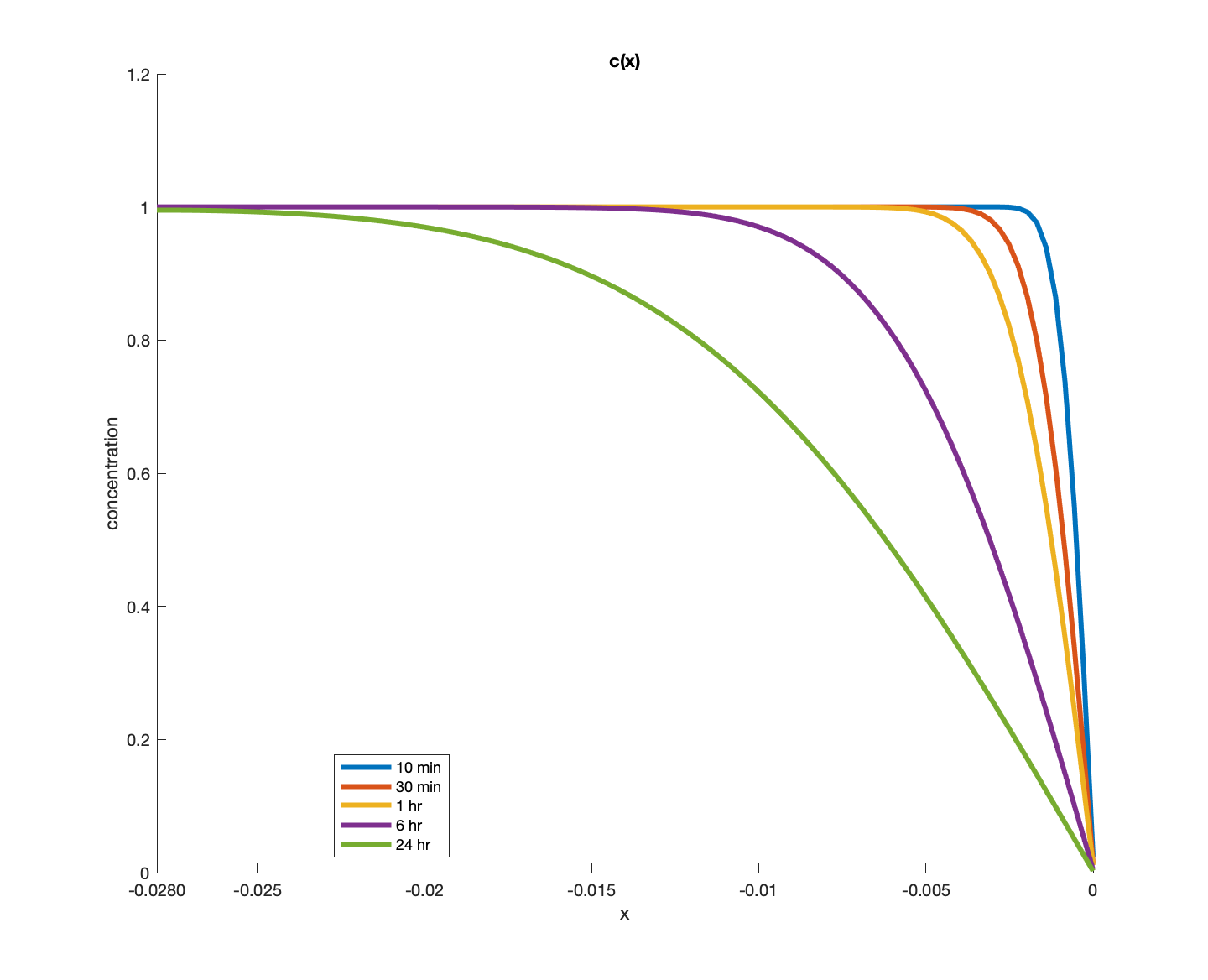}
	  \caption{Stent concentrations.}
	\label{fig:snapsA}
	\end{figure}
	
	\begin{figure}[h]
	  \centering
	 	 \includegraphics[width=0.9\linewidth]{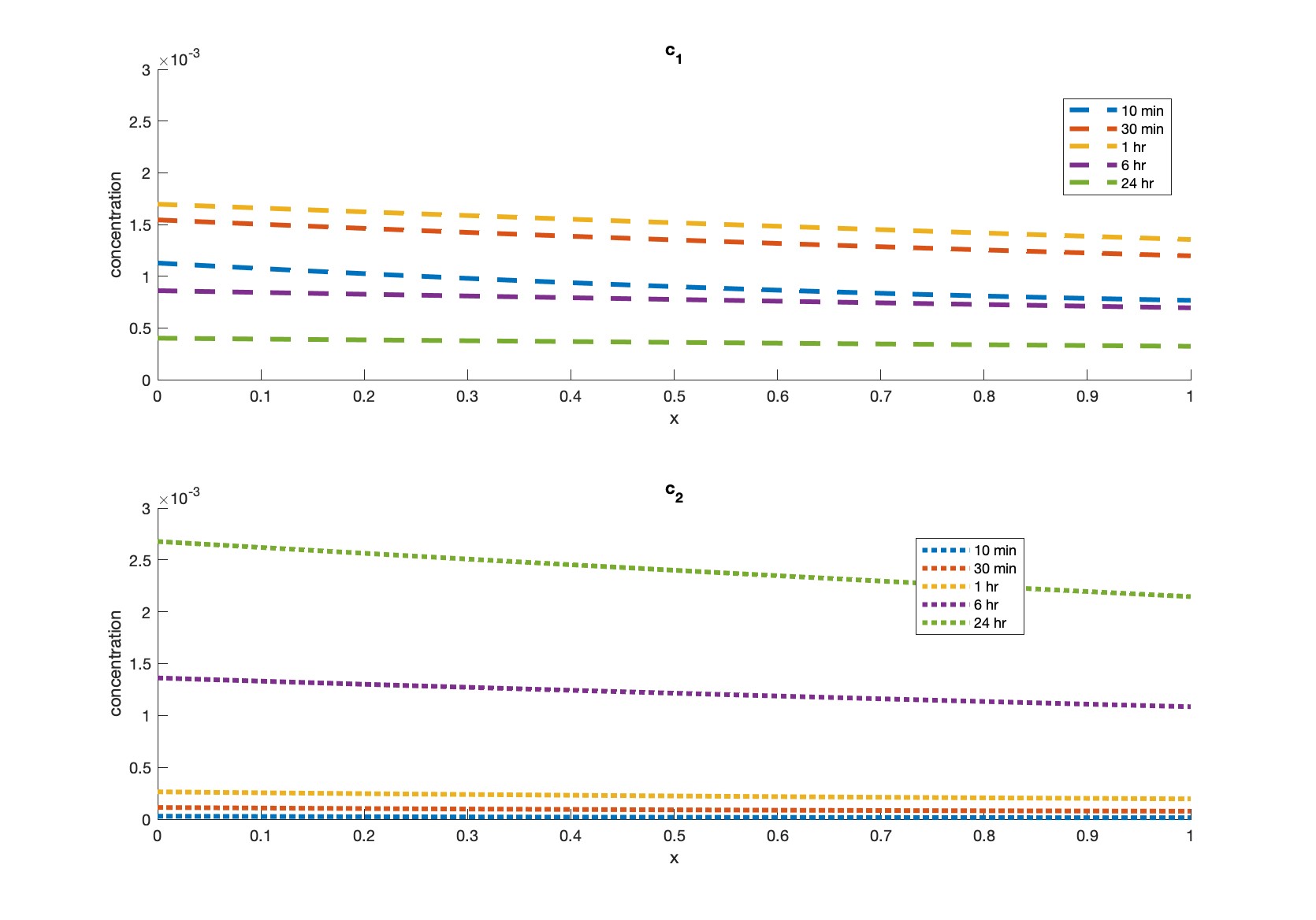}
	  \caption{Wall concentrations.}
	\label{fig:snapsB}
	\end{figure}

	\begin{figure}[h]
	    \centering
	    \includegraphics[width=\textwidth]{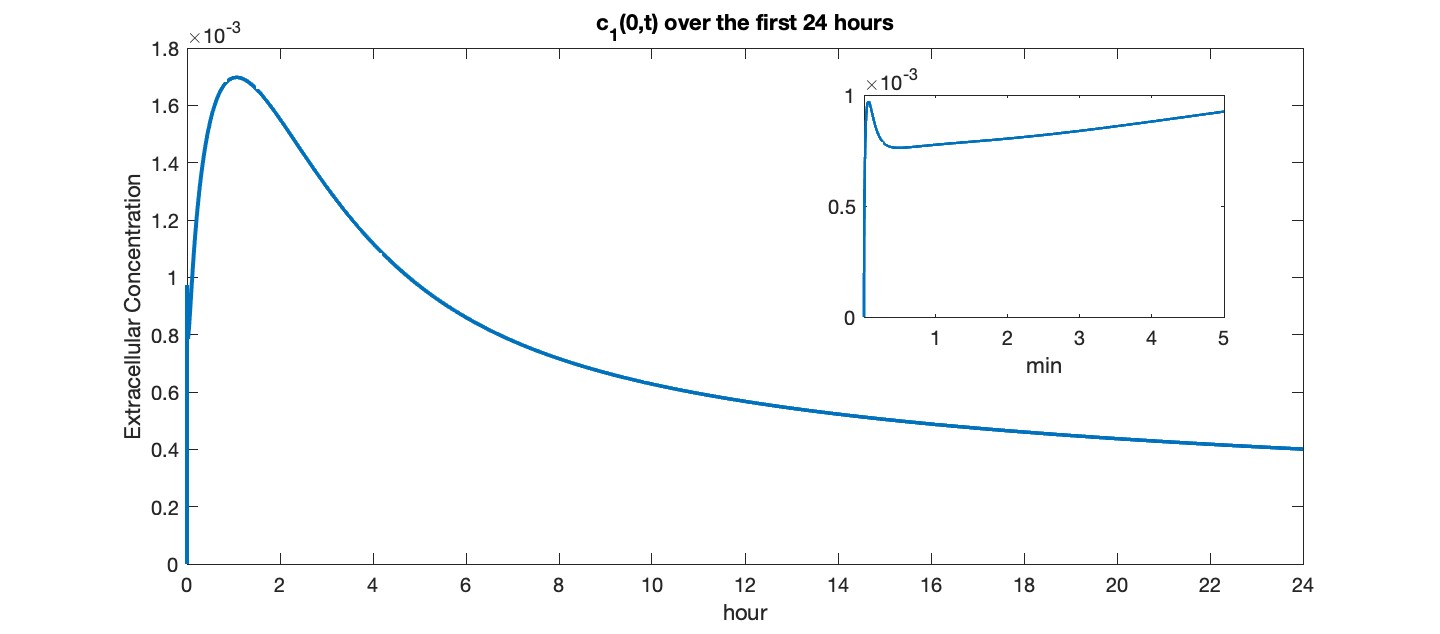}
	\caption{Interface concentration: extracellular.}
	\label{fig:interface1}
	\end{figure}

	The finite-element simulation results presented in this section were also confirmed by those obtained via a finite-difference method (which are not presented here to save space). The $\Linf(L^2)$-error between the finite-element and the finite-difference results was found to be around $5\x 10^{-5}$ when $h_\Os = 2.8\x 10^{-4}$ and $h_\Om = 0.01$.

%%%%%%%%%%%%%%%%%%%%%%%%%%%%%%%%%%%%%% 
\section{Two-dimensional generalization}\label{sec-6}
Below, we propose a two-dimensional drug-release model which is an extension of the one-dimensional model studied in the previous sections. \Cref{2d} shows a two-dimensional sketch of the arterial wall. 

Again, as in the one-dimensional case, $c,c_1,$ and $c_2$ represent the unknown concentrations in the stent, in the extracellular matrix, and in the muscle cells, respectively. All other quantities are given coefficients. 
Unlike the one-dimensional model, we note that there are four additional pieces of the  boundary. New boundary conditions must be prescribed there. For these four pieces, Dirichlet boundary conditions were imposed on $\Gam^a$ and $\Gam^b$, whereas the periodic boundary conditions were used on $\Gam^c$ and $\Gam^d$ for $c_1$.

\begin{figure}\centering
\begin{tikzpicture}[scale=0.65]
\draw[fill=black!10,even odd rule] (0,0) circle (1.5 cm) circle (2 cm);
\node at(0,-0.5) {Lumen};
\node at(0,1.7) {Wall};
\draw [fill=black!30] (-0.3,1.3) rectangle (0.3,1.5);
\draw [fill=black!30] (-0.3,-1.3) rectangle (0.3,-1.5);
\draw [fill=black!30] (-1.5,-0.3) rectangle (-1.3,0.3);
\draw [fill=black!30] (1.5,-0.3) rectangle (1.3,0.3);
%\node at (-0.5,0.6) {Stent};
%\draw[->] (-0.3,0.75) to (0,1.3);
%\draw[->] (-0.7,0.35) to (-1.3,0);
\node at (-3,1) {Stent};
\draw[->] (-2.5,1) to (0,1.4);
\draw[->] (-2.5,1) to (-1.4,0);
\draw [thick, red] (1.6,0) circle (0.6cm);
\draw [dotted, very thick, blue] (0,0) -- (1.3,0);
%\draw [thick, blue] (1.3,0) -- (2,0);
\filldraw [black] (0,0) circle (1pt);
%\filldraw [blue] (1.3,0) circle (1.5pt);
%\filldraw [blue] (1.5,0) circle (1.5pt);
%\filldraw [blue] (2,0) circle (1.5pt);
\draw [blue, thick] (1.3,-0.3) rectangle (2,0.3);
\draw[->,red,thick] (1.7,0.35) to[bend left] (4.5,0.35);
%

%%%
%%%
\filldraw [line width = 0.3mm, fill=black!30] (5,-1.5) -- (6.5,-1.5) -- (6.5,1.5) -- (5,1.5) -- (5,-1.5);
\filldraw [line width = 0.3mm, fill=black!10] (6.5,-1.5) -- (11,-1.5) -- (11,1.5) -- (6.5,1.5) -- (6.5,-1.5);
%\draw [line width = 0.3mm] (6.5,-1.5) -- (6.5,1.5);
%
\node at (5.8,0) {$\bm\Os$};
\node at (8.75,0) {$\bm\Om$};
\node at (4.7,-0.7) {\textcolor{blue}{$\Gam^s$}};
\node at (6.7,-0.7) {\textcolor{blue}{$\Gam$}};
\node at (11.4,-0.7) {\textcolor{blue}{$\Gam^m$}};
\node at (5.8,1.8) {\textcolor{blue}{$\Gam^b$}};
\node at (5.8,-1.8) {\textcolor{blue}{$\Gam^a$}};
\node at (8.75,1.8) {\textcolor{blue}{$\Gam^d$}};
\node at (8.75,-1.8) {\textcolor{blue}{$\Gam^c$}};
\end{tikzpicture}
\caption{Arterial wall; schematic diagram in 2-d.}
\label{2d}
\end{figure}
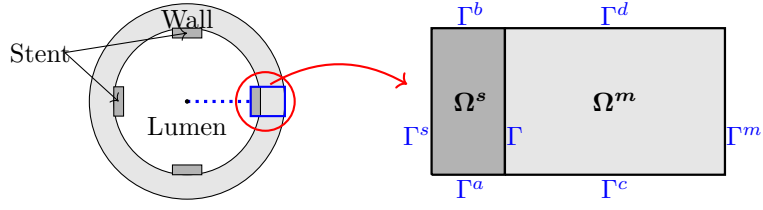

Essentially, all of the PDE and numerical analysis results for the 1-d model can be extended to the 2-d model; this includes a priori estimates and well-posedness proofs, as well as the finite-element convergence and error estimates. One notable difference is that the boundary of the 2-d domain consists of 1-d line segments; handling the 2-d  boundary terms requires some additional and more involved trace inequalities. Below, we only state the key formulations and main results, without giving detailed derivations and proofs, to save space. However, it should be noted that, although the PDE and numerical analyses are similar, the computer simulations and coding are much more complicated in the 2-d case; we will present those results elsewhere. 

Our two-dimensional model is given as follows:
%\textbf{Modified 2D System (with dimensions/units):}
\begin{alignat}{3}
   \p_t c - \text{div}\big(D(x) \nab c\big) &= 0,  								&&x\in\Os,          &&t\in(0,T], \label{2dPDE}\\
   D(x) \nab c \cd \bm{n} + \be c &= \be c_1,  								&&x\in\Gam,       &&t\in(0,T], \label{2dBC}\\   
   D(x) \nab c \cd \bm{n} &= 0,  											&&x\in\Gam^s, \,&&t\in(0,T], \label{2dBCl}\\
   c(x,t) &=  a(x,t),   													&&x\in\Gam^a, \,&&t\in(0,T], \label{2dBCa}\\
   c(x,t) &=  b(x,t),   													&&x\in\Gam^b, \,&&t\in(0,T], \label{2dBCb}\\
   c &= c_0,  														&&x \in \overline\Os, &&t=0; \label{2dIC}\\[8pt]
   \phi \p_t c_1+ \bm v \cd \nab c_1 - \text{div}\big(D_1(x) \nab c_1\big) +\al c_1 &=  \frac{\al}{\kap} c_2,  
   																&&x\in\Om,       &&t\in(0,T] , \label{2dPDEi}\\
   D_1(x) \nab c_1 \cd \bm{n} - \bm{v}\cd\bm{n} c_1 &= D(x) \nab c \cd \bm{n},  &\quad&x\in\Gam,     &&t\in(0,T] , \label{2dBCi}\\
   D_1(x) \nab c_1 \cd \bm{n}_1 &= 0,  									&&x\in\Gam^m, &&t\in(0,T] ,  \label{2dBCir}\\
   c_1(x) - c_1(x+L_2e_2)  &= 0,   										&&x\in\Gam^c, &\quad&t\in(0,T],  \label{2dBCi1}\\
   \nab c_1(x) \cd \bm{n}_1 - \nab c_1(x+L_2e_2) \cd \bm{n}_1 &= 0,  			&&x\in\Gam^c, &\quad&t\in(0,T],  \label{2dBCi1n}\\
   c_1 &= 0,  														&&x \in \overline\Om, &&t=0;  \label{2dICi} \\[8pt]
    (1-\phi) \p_t c_2 + \frac{\al}{\kap } c_2 &= \al c_1,  			&&x\in\Om, \,&&t\in(0,T] , \label{2dODEii}\\
    c_2 &= 0,  														&&x \in \overline\Om, \,&\quad&t=0. \label{2dICii}
\end{alignat}
We note that \eqref{2dBCi1} and \eqref{2dBCi1n} represent the periodic boundary conditions for $c_1$.

\begin{theorem}[Well-posedness in 2-d]\label{2-dwellposedness}
Under some reasonable assumptions on the coefficients, there exists a unique weak solution $(c,c_1,c_2)$ satisfying the following for any $T>0$:
\begin{gather*}
c \in L^\infty(0,T;L^2(\Os)) \cap L^2(0,T;H^1(\Os)) \cap H^1(0,T;H^{-1}(\Os)),\\
c_1 \in L^\infty(0,T;L^2(\Om)) \cap L^2(0,T;H_{per}^1(\Om)) \cap H^1(0,T;H^{-1}(\Om)),\\
c_2 \in {\color{black} H^1(0,T;L^2(\Om)), }
\end{gather*}
and, for $t\in (0,T]$,
\begin{alignat*}{2} 
\< \p_t c, v \>_{\V^{s'}\x \V^s} + \calA[c, v; t] &= \< \be c_1, v \>_{\Gam} &\quad& \forall \,v\in H^1(\Os),   \\
\< \phi \p_t c_1,w \>_{\V^{m'}\x\V^m} + \calB[c_1,w;t] &= \<  \be c, w \>_{\Gam} + \textstyle\frac{\al}{\kap } (c_2,w)_{\Om} && \forall \,w\in H_{per}^1(\Om),   \\
(1-\phi) \p_t c_2 + \textstyle\frac{\al}{\kap} c_{2} &= \al c_{1} &\qquad& a.e. \,x\in\Om, \\
c(\cd,0)=c_0(x)\in L^2(\Os), \quad c_1(\cd,0) &\equiv c_2(\cd,0) \equiv 0.
\end{alignat*}
%with the initial conditions $c(\cd,0)=c_0(x)\in L^2(\Os), c_1(\cd,0) \equiv c_2(\cd,0) \equiv 0$.
Here, $$H^1_{per}(\Omega^m):= \{u \in H^1: u(x) = u(x + L_2e_2) \text{ and } \nab u(x) = \nab u(x + L_2e_2), \, a.e. x\in \Gam^c \},$$
$(\cd,\cd)$ denotes the $L^2$-inner product, $\<\cd,\cd\>$ denotes the duality pairing, and
\begin{align*}
\calA[u,v;t] &:=  \big(D\nab u, \nab v\big)_{\Os} + \<  \be u, v \>_{\Gam}, \\
\calB[u,w;t] &:= \big( D_1\nab u, \nab w \big)_{\Om} + \big( \bm v \cd \nab u, w \big)_{\Om} + \al \big( u, w \big)_{\Om}  + \big\< (\be - \bm{v}\cd\bm{n}_1) u, w \big\>_{\Gam}.
\end{align*}
\end{theorem}

Before introducing the finite-element result, we first introduce the finite-element  spaces and the concept of approximate weak solutions.
Let $\mathcal T^s_h$ be a mesh of $\Os$ and $\mathcal T^m_h$ a mesh of $\Om$, both geometrically conformal. Let 
\begin{align*}
	\V_h^s &:= \{ v_h \in \V^m; v_h|_K \in P_1(K)\, \, \forall K \in \mathcal T^s_h\},\\ \V_h^m &:= \{ w_h \in \V^m; w_h|_K \in P_1(K)\,\, \forall K \in \mathcal T^m_h\}, \\ \V_{h,per}^m &:= \{w_h \in \V_h^m; w_h \text{ is periodic along the $x_2$ direction}\}. 
\end{align*}
It is a known fact that each $\V_h^s$ and $\V_h^m$ has a set of nodal basis functions, denoted by $\{\psi_j^s\}$ and $\{\psi_j^m\}$, respectively, and satisfying that $\psi_j^s(\bm a_i) = \del_{ij}$ and $\psi_j^m(\hat{\bm a_i}) = \del_{ij}$ for each node $\bm a_i$ and $\hat{\bm a_i}$, and each $j$.

\begin{definition} \label{FEM2d}
 $(c_h,c_{1h}, c_{2h}): (0,T] \rightarrow \V^s_h \times \V^m_{h,per} \times \V^m_h$ is called an approximate weak solution to the 2-d system if 
\begin{alignat*}{2}
\< \p_t c_h, v_h \>_{\V^{s'}\x\V^s} + \calA[c_h, v_h;t] &= \mathcal{F}[c_{1h}, v_h;t] &&\quad \forall \,v_h \in \V^s_h, \\
 \< \phi \p_t c_{1h}, w_h \>_{\V^{m'}\x\V^m} + \calB[c_{1h}, w_h;t] &= \mathcal{F}[c_h, w_h;t] + \frac{\al}{\kap} (c_{2h}, w_h)_{\Om} && \quad \forall \,w_h \in \V^m_h, \\
(1-\phi) \p_t c_{2h} + \frac{\al}{\kap} c_{2h} &= \al c_{1h}. &&
\end{alignat*}
with $c_h(0,\cd)=\calPh c_0 \in \V^s_h$ and $c_{1h}(0,\cd) \equiv c_{2h}(0,\cd) \equiv 0$.
\end{definition}

We obtained the following well-posedness and error estimate results. 

	{\color{black} 
\begin{theorem}[Error estimates in 2-d]
For each $h>0$, there exists a unique approximate weak solution $(c_h,c_{1h}, c_{2h})$. 
Moreover, let $(c,c_1,c_2)$ be the weak solution stated in Theorem \ref{2-dwellposedness} and define error functions $e_h := c-c_h$, $e_{1h} := c_1-c_{1h}$, and $e_{2h} := c_2-c_{2h}$, then, the following error estimates hold:

	\begin{align} \label{2Derror_estimate1}
		\|e_h\|_{\Linf(0,T;L^2(\Os))} + \|e_{1h}\|_{\Linf(0,T;L^2(\Om))} + \|e_{2h}\|_{\Linf(0,T;L^2(\Om))} &\les h^2,\\
	 \|e_h\|_{L^2(0,T;H^1(\Os))} + \|e_{1h}\|_{L^2(0,T;H^1(\Om))} 
	 &\les h. \label{2Derror_estimate2}
	\end{align}
\end{theorem}

We note that the interface terms now appear as $L^2$ integrals on $\Gamma$ in the 2-d case, which only gives us the control of $\|e_h\|_{L^2(0,T;L^2(\Gamma))}$ and $\|e_{1h}\|_{L^2(0,T; L^2(\Gamma))}$, not pointwise control as in the 1-d case. But, these estimates are consequences of \eqref{2Derror_estimate2} and the trace inequality. 
}

%We fully expect the fully discrete error to be first order in space with the Euler method. 

%%%%%%%%%%%%%%%%%%%%%%%%%%%%%%%%%%%%%% 
\section{Summary and concluding remarks}\label{sec-7}
{\color{black} 
In this paper we have presented a complete PDE and numerical analysis for the modified one-dimensional intravascular stent model originally proposed in \cite{McGinty13}. The modified model is not only well-posed, it also has improved numerical computability. The well-posedness was proved by using the Galerkin method in combination with a compactness argument. A semi-discrete finite-element method and a fully discrete scheme using the Euler time-stepping was formulated for the PDE model. Optimal order error estimates in the energy norm was proved for both schemes. Numerical results have been presented, along with comparisons between different decoupling strategies and time-stepping schemes. Finally, extensions of the model, along with its PDE and numerical analysis results for the two-dimensional case, were also briefly discussed.

Our future work on this research project will focus on the direct generalizations of this model in one and more dimensions in the modeling of the transmural advection using Darcy's flow; the model will also include an additional lumenal layer that considers the effect of blood flow on drug delivery. Moreover, we are very much interested in completing higher-dimensional codes and simulations as well as in validating the model simulations with real lab data. 

\textcolor{black}{It is also worth noting that the analysis techniques employed in this work should be readily extendable to other more sophisticated linear models. In addition, there have been recent developments in nonlinear drug-binding models (see \cite{McGinty22IJP} and \cite{McGinty22JCR}). In those cases,  the nonlinear terms require special care and new techniques. It is expected that our general techniques might still be applicable. Such a detailed analysis will be carried out as future work.}
}

%%%%%%%%%%%%%%%%%%%%%%%%%%%%%%%%%%%%%
%%%%%%%%%%%%%%%%%%%%%%%%%%%%%%%%%%%%%

\appendix
\section{Parameters and cited facts}\label{sec-8}

%%%
\subsection{Parameters}\label{sec-8-1}
We summarize in the chart below  the parameter values that appeared  in the one-dimensional model and used in our simulations. 
We refer the reader to \cite{McGinty13} for more details. 

\medskip
\begin{center}
\begin{tabular}{ |c|c|c|c| } 
 \hline
Parameter & Symbol & Value \\ %& Unit \\ 
 \hline\hline
Media porosity & $\phi$ & 0.61  \\ 
 \hline
Partition coefficient & $K$ & 15   \\ 
 \hline
                    & $\del$ & $4\x10^{-7}$   \\ 
 \hline
                    & $l$ & 0.028   \\
 \hline
                    & $\tP$ & $4.5\x10^{4}$   \\
  \hline
Peclet number   & $\Pe$ & 0.1044   \\
  \hline
Damkholer number   & $\Da$ & 0.0162  \\
  \hline
\end{tabular}
\end{center}
\medskip

%%%
\subsection{Some known facts}\label{sec-8-2}

In this subsection, we present a few well-known lemmas and theorems cited in this paper and either provide a brief proof or cite at least one reference where a proof can be found.

% ************************************
\begin{lemma}[A 1-d trace inequality]	\label{trace_1d}
Let $v: [a,b]\to \R$, $v\in H^1(a,b)$, and $H:=b-a$. Then
\[ 
|v(a)| \leq C \|v\|_{H^1(a,b)} \quad\mbox{and}\quad  |v(b)| \leq C \|v\|_{H^1(a,b)}. 
\]
where $C \approx 1$ if $H\gg1$ and $C \in \calO \big( H^{-1/2} \big)$ if $H\ll1$. 

\end{lemma}

\begin{proof}
Let $w: [a,b]\to \R$ be differentiable. By the fundamental theorem of calculus we have the following for any $a\leq w_1<w_2\leq b$:
\[ w(x_2) = w(x_1) + \int_{x_1}^{x_2} w'(s) \,ds. \]
Setting $w(x):=[(x-a)v(x)]^2$, $x_1:=a$, and $x_2:=b$ yields
\begin{align*}
[hv(b)]^2 &= [(a-a)v(a)]^2 + \int_a^b \Big((x-a)^2v^2(x) \Big)' \, dx \\
&= \int_a^b 2(x-a)v^2(x) + 2(x-a)^2v(x)v'(x) \, dx \\
%&\leq  \int_a^b 2hv^2(x) + h^22v(x)v'(x) \, dx \\
%&\leq  \int_a^b 2hv^2(x) + h^2v^2(x) + h^2(v'(x))^2 \, dx  \\
&\leq (2H+H^2) \|v\|^2_{L^2(a,b)} + H^2\|v'\|^2_{L^2(a,b)}.
\end{align*}
Therefore, 
\begin{align*}
  v^2(b) &\leq (2/H+1) \|v\|^2_{L^2(a,b)} + \|v'\|^2_{L^2(a,b)} ,
 %   |v(b)| &\leq (2h^{-1}+1)^{1/2} \|v\|_{H^1(a,b)}. 
\end{align*}
which gives the first inequality.  The second inequality follows similarly with $w(x):=[(b-x)v(x)]^2$. 
%
%Similarly with $w(x):=[(b-x)v(x)]^2$, $x_1:=a$ and $x_2:=b$,
%\begin{align*}
%[hv(a)]^2 &= [(b-b)v(b)]^2 - \int_a^b \Big((b-x)^2v^2(x) \Big)' \, dx 
%\leq (2h+h^2) \|v\|^2_{L^2(a,b)} + h^2\|v'\|^2_{L^2(a,b)}
%\end{align*}
%
%Hence $|v(a)| \leq (2h^{-1}+1)^{1/2} \|v\|_{H^1(a,b)}$. This concludes the proof.
\end{proof}

% ************************************
\begin{lemma}[General form of Gr\"onwall's inequality; see Appendix B2 of \cite{Evans}] \label{gronwall}
Let $\xi,\phi,\psi \in L^2(0,T)$ and be nonnegative, and let $\eta \in H^1(0,T)$. If 
\[ 
\eta'(t)  + \xi(t) \leq \phi(t) \eta(t) + \psi(t) \qquad \forall t\in (0,T],
 \]
then 
\[ 
\eta(t) + \int_0^t \xi(s) ds \leq  \exp\Bigl\{\int_0^t \phi(s) ds \Bigr\}  \( \int_0^t  \psi(s) ds + \eta(0) \)  \qquad \forall  t\in(0,T].
\] 
   
\end{lemma}

%\begin{proof}
%This is a mild generalization of the Gr\"onwall's Inequality in \cite{Evans} Appendix B2. First observe that 
%\begin{align*}
%\frac{d}{ds} \left[ \(\eta(s) + \int_0^s \xi(r) dr \) e^{-\int_0^s \phi(r) dr} \right]  
%	&=  e^{-\int_0^s \phi(r) dr} \left[ \eta'(s) + \xi(s) - \phi(s) \( \eta(s) + \int_0^s \xi(r) dr \) \right] \\
%	&\leq  e^{-\int_0^s \phi(r) dr} \left[ \psi(s) - \phi(s)  \int_0^s \xi(r) dr  \right] \\
%	&\leq  e^{-\int_0^s \phi(r) dr}  \psi(s) 
%\end{align*}
%where the last inequality follows from the non-negativity of $\phi,\xi$. Next integrate the above inequality on $[0,t]$ for any $t\in(0,T]$ where $\eta(t)$ is well-defined,
%\begin{align*}
%\eta(t)e^{-\int_0^t \phi(r) dr} + \(\int_0^t \xi(r) dr \)e^{-\int_0^t \phi(r) dr}   
%	&\leq  \eta(0) + \int_0^t e^{-\int_0^s \phi(r) dr}  \psi(s)  ds \\
%	&\leq  \eta(0) + \int_0^t  \psi(s)  ds
%\end{align*}
%where the last inequality again is due to the non-negativity of $\phi$. This concludes the proof.
%\end{proof}

% ************************************
\begin{theorem}[Theorem 5.9 of \cite{Evans}] \label{evans_more_calculus}
Suppose that ${u} \in L^2(0,T;H^1_0(U))$ and ${u}' \in L^2(0,T;H^{-1}(U))$.
\begin{itemize}
	\item[(i)] Then, ${u} \in C([0,T];L^2(U))$ (after possibly being redefined on a set of measure zero).
   \item[(ii)] The mapping $t \mapsto \|{u}(t)\|^2_{L^2(U)}$ is absolutely continuous, with 
   \[ 
   \frac{d}{dt} \|{u}(t)\|^2_{L^2(U)} = 2\< {u}'(t), {u}(t) \>  \qquad \mbox{for a.e. } 0\leq t\leq T.
   \]
   \item[(iii)] Furthermore, the following holds 
\[
 \max_{0\leq t\leq T}\|{u}(t)\|_{L^2(U)} \leq C \Big( \|{u}\|_{L^2(0,T;H_0^1(U))} + \|{u}'\|_{L^2(0,T;H^{-1}(U))}  \Big) . 
 \]
\end{itemize}
\end{theorem}

% ************************************
\begin{lemma}[Lemma 1.1, Chapter IV of \cite{Hartman}]\label{hartman}
The initial value problem given by
\[ 
\bm{y}\,'(t) = Q(t) \,\bm{y}(t), \quad \bm{y}(0) = \bm{y}_0 
\] 
has a unique solution $\bm{y}:[0,T]\to \bm{R}^n$  if $Q$ is a continuous function on $[0,T]$. 
\end{lemma}

% ************************************
\begin{theorem}[Theorem 4.10.8 of \cite{Friedman}]\label{weak_seq_compact}\label{friedman1}
Let $X$ be a reflexive Banach space. A set $K\subset X$ is weakly sequentially compact if and only if it is both bounded and weakly closed. 
\end{theorem}

%\begin{remark}
%In the lemma above, the weakly closed condition is only used to show that the weak limit is in $K$, but not needed to ensure existence of a convergent subsequence. 
%\end{remark}

% ************************************
\begin{theorem}[Theorem 4.12.3 of \cite{Friedman}]\label{weakly_subseq}\label{friedman2}
Let $X$ be a separable normed linear space. Then every bounded sequence of continuous linear functionals in $X^*$ has a weakly convergent subsequence. 
\end{theorem}

% ************************************
\begin{theorem}[Theorem 5.1.1 of \cite{Friedman}]\label{compact_op}\label{friedman3}
A completely continuous (compact) linear operator maps weakly convergent sequences into convergent sequences. 
\end{theorem}


\begin{thebibliography}{99}
\bibliographystyle{abbrv}
    
%\bibitem{Balakrishnan}
%\newblock {B. Balakrishnan, J. F. Dooley, G. Kopia, E. R. Edelman}, 
%\newblock {Intravascular drug release kinetics dictate arterial drug deposition, retention, and distribution}, {\em J. Controll. Release}, \textbf{123} (2007), 100--108.
%\doilink{}

\bibitem{blausen}
\newblock {Blausen.com staff (2014)}, 
\newblock {\href{https://en.wikiversity.org/wiki/WikiJournal_of_Medicine/Medical_gallery_of_Blausen_Medical_2014}{Medical gallery of Blausen Medical 2014}}.  
\newblock {WikiJournal of Medicine},\textbf{1} (2) (2014).
%\doilink{https://doi.org/10.15347/wjm/2014.010}



\bibitem{PontrelliDeMonte07}
\newblock {G. Pontrelli, F. Monte}, 
\newblock {Mass diffusion through two-layer porous media: an application to the drug-eluting stent}, 
\newblock {Journal of Heat and Mass Transfer}, \textbf{50} (2007), 3658--3669.
%\doilink{https://doi.org/10.1016/j.ijheatmasstransfer.2006.11.003}



\bibitem{McGinty11}
\newblock {S. McGinty, S. McKee, R. M. Wadsworth, C. McCormick}, 
\newblock {Modeling drug-eluting stents} 
\newblock {Math. Med. Biol.}, \textbf{28} (2011), 1--29.
%\doilink{https://doi.org/10.1093/imammb/dqq003}



\bibitem{McGinty13}
\newblock {S. McGinty, S. McKee, R. M. Wadsworth, C. McCormick}, 
\newblock {Modeling arterial wall drug concentrations following the insertion of a drug-eluting stent}, 
\newblock {SIAM J. Appl.. Math.}, \textbf{73}, No. 6 (2013), 2004--2028.
%\doilink{https://doi.org/10.1137/12089065X}



\bibitem{Balakrishnan}
\newblock {B. Balakrishnan, J. F. Dooley, G. Kopia, E. R. Edelman}, 
\newblock {Intravascular drug release kinetics dictate arterial drug deposition, retention, and distribution}, 
\newblock {J. Controll. Release}, \textbf{123} (2007), 100--108.
%\doilink{https://doi.org/10.1016/j.jconrel.2007.06.025}



\bibitem{Borghi.et.al}
\newblock {A. Borghi, E. Foa, R. Balossino, F. Migliavacca, G. Dubin}i, 
\newblock {Modelling drug elution from stents: effects of reversible binding in the vascular wall and degradable polymeric matrix},  
\newblock {Computer Methods in Biomechanics and Biomedical Engineering}, \textbf{11}, No. 4 (2008), 367--377.
%\doilink{https://doi.org/10.1080/10255840801887555}



\bibitem{ECPMM2020}
\newblock {J. Escuer, M. Cebollero, E. Pe\~na, S. McGinty, M. A. Mart\'inez}, 
\newblock {How does stent expansion alter drug transport properties of the arterial wall}, 
\newblock {Journal of the Mechanical Behavior of Biomedical Materials}, \textbf{104} (2020), Article 103610.
%\doilink{https://doi.org/10.1016/j.jmbbm.2019.103610}



\bibitem{Feenstra}
\newblock {P. Feenstra, C. Taylor}, 
\newblock {Drug transport in artery walls: a sequential porohyperelastic-transport approach}, 
\newblock {Computer Methods in Biomathematics and Biomedical Engineering}, \textbf{12}, No. 3 (2009), 263--276.
%\doilink{https://doi.org/10.1080/10255840802459396}



\bibitem{Grassi.et.al}
\newblock M. Grass et. al., 
\newblock {Novel design of drug delivery in stented arteries: a numerical comparative study}, 
\newblock {Math. Biosci. Engineering}, \textbf{6}, No. 3 (2009), 493--508.
%\doilink{https://doi.org/10.3934/mbe.2009.6.493}



\bibitem{Horner}
\newblock {M. Horner, S. Joshi, V. Dhruva, S. Sett, S. F. C. Stewart}, 
\newblock {A two-species drug delivery model is required to predict deposition from drug-eluting stents}, 
\newblock {Cardiovascular Engineering and Technology}, \textbf{1}, No. 3 (2010), 225--234.
%\doilink{https://doi.org/10.1007/s13239-010-0016-4}



\bibitem{Hose.et.al}
\newblock {D. R. Hose, A. J. Narracott, D. Griffiths, S. Mahmood, J. Gunn, D. Sweeney et. al.}, 
\newblock {A thermal analogy for modelling drug elution from cardiovascular stents}, 
\newblock {Computer Methods in Biomechanics and Biomedical Engineering}, \textbf{7}, No. 5 (2004), 257--264.
%\doilink{https://doi.org/10.1080/10255840412331303140}



\bibitem{WSR}
\newblock {J. M. Weiler, E. M. Sparrow, R. Ramazani}, 
\newblock {Mass transfer by advection and diffusion from a drug-eluting stent},  
\newblock {International Journal of Heat and Mass Transfer}, \textbf{55} (2012), 1-7.
%\doilink{https://doi.org/10.1016/j.ijheatmasstransfer.2011.07.020}



\bibitem{Zunino04}
\newblock {P.~Zunino}, 
\newblock {Multidimensional pharmacokinetics models applied to the design of drug-eluting stents}, 
\newblock {Cardiovasc. Eng.: Int. J.}, \textbf{4} (2004), 181--191.
%\doilink{https://doi.org/10.1023/B:CARE.0000031547.39178.cb}



\bibitem{McGinty14}
\newblock {S.~McGinty}, 
\newblock {A decade of modelling drug release from arterial stents}, 
\newblock {Mathematical Biosciences}, \textbf{257} (2014), 80--90.
%\doilink{https://doi.org/10.1016/j.mbs.2014.06.016}

\bibitem{BrennerScott} 
\newblock {S. C. Brenner, L. R. Scott} 
\newblock {The Mathematical Theory of Finite Element Methods}, 
\newblock {Springer-Verlag}, New York, NY (2002).
%\doilink{https://doi.org/10.1007/978-0-387-75934-0}



\bibitem{McGinty22IJP}
\newblock {J. Escuer, A. F. Schmidt, E. Pe\~na, M. A. Mart\`inez, S.~McGinty}, 
\newblock {Mathematical modelling of endovascular drug delivery: Balloons versus stents}, 
\newblock {International Journal of Pharmaceutics}, \textbf{620} (2022), Article 121742.
%\doilink{https://doi.org/10.1016/j.ijpharm.2022.121742}



\bibitem{McGinty22JCR}
\newblock {A. McQueen, J. Escuer, A. F. Schmidt, A. Aggarwal, S. Kennedy, C. McCormick et. al.}, 
\newblock {An intricate interplay between stent drug dose and release rate dictates arterial restenosis}, 
\newblock {Journal of Controlled Release}, \textbf{349} (2022), 992--1008.
%\doilink{https://doi.org/10.1016/j.jconrel.2022.07.037}



\bibitem{Evans} 
\newblock {L. C. Evans}, 
\newblock {Partial Differential Equations},  
\newblock {AMS}, Providence, RI (2010).
%\doilink{https://doi.org/10.1090/gsm/019}



\bibitem{Hartman}
\newblock {P. Hartman}, 
\newblock {Ordinary Differential Equations}, 
\newblock {SIAM}, Philadelphia (1964).
%\doilink{https://doi.org/xxxxx}



\bibitem{Friedman}
\newblock {A. Friedman}, 
\newblock {Foundations of Modern Analysis}, 
\newblock {Dover}, New York (2010).
%\doilink{https://doi.org/xxxxx}

 

\end{thebibliography}
\end{document}